\documentclass{amsart}
\usepackage[latin1]{inputenc}
\usepackage{amsmath,amssymb,graphicx,setspace,verbatim}
\usepackage{color}
\usepackage{comment}
\usepackage{appendix}
\theoremstyle{comment}
\usepackage[color,matrix,arrow]{xy}
\usepackage{cleveref}
\newtheorem*{mcomment}{\color{cyan}{Comment}}

\usepackage{multicol}
\usepackage{multirow}
\input xy
\xyoption{all}
\newtheorem{theorem}{Theorem}[section]

\newtheorem{lemma}[theorem]{Lemma}
\newtheorem{corollary}[theorem]{Corollary}
\newtheorem{proposition}[theorem]{Proposition}
\theoremstyle{definition}

\begin{document}

\title{Faithful permutation representations of toroidal regular maps}

\author{Maria Elisa Fernandes}
\address{Maria Elisa Fernandes, Department of Mathematics,
University of Aveiro,
Aveiro,
Portugal}
\email{maria.elisa@ua.pt}

\author{Claudio Alexandre Piedade}
\address{Claudio Alexandre Piedade, Department of Mathematics,
University of Aveiro,
Aveiro,
Portugal}
\email{claudio.a.piedade@ua.pt}

\begin{abstract}
In this paper we list all possible degrees of a faithful transitive
permutation representation of the group of symmetries of a regular map of
types $\{4,4\}$ and $\{3,6\}$ and we give examples of graphs, called CPR-graphs, representing some of these permutation representations. 
\end{abstract}

\maketitle


\noindent \textbf{2010 Math Subj. Class:} 52B11, 05E18,  20B25.

\section{Introduction}

The classification of highly symmetric objects, particularly regular maps and polytopes, is a problem that attracts both geometers and algebraists.
The idea of using permutation representations to classify regular maps and polytopes is not new but in  2008 the concept of a graph associated with a regular polytope, called a CPR graph,  was introduced \cite{CPR}. CPR graphs, that are faithful permutation representations, turned out to be a powerful tool in classification of abstract regular polytopes.
The group of symmetries $G$ of an abstract regular polytope of rank $r$ is generated by $r$ involutions $\rho_0,\ldots, \rho_{r-1}$.
  For each $i\in\{0,\ldots,r-1\}$ the $i$-faces correspond to the cosets of the group $G_i$ generated by all the generators of $G$ except $\rho_i$ \cite{ARP}. The group $G$ acts transitively on the set of $i$-faces and if, in addition, $G_i$ is core-free this action is faithful. Regular maps that have non-faithfull actions on the cells (vertices, edges or faces) or on darts are identified in \cite{Siran2005}. 

In this paper we search for faithful transitive permutation representations of the toroidal maps of type $\{4,4\}$ and $\{3,6\}$, which includes the faithful actions on vertices, edges and faces of the map. 

CPR graphs give an efficient method to classify abstract regular polytopes   for a given a group of symmetries. Several results were accomplished  using these faithful permutation representations answering some conjectures arising from the atlas of abstract regular polytopes for small groups built in 2006 by  Leemans and Vaulthier \cite{LVAtlas}. For instance, the symmetric group $S_n$ ($n\geq 4$) is the group of symmetries of a polytope of any rank between $3$ and  $n-1$  \cite{FL1}.
CPR graphs were used to construct examples of polytopes for all possible ranks, furthermore to prove that there are exactly two polytopes of rank $r=n-2$ for $S_n$ and to describe all polytopes of rank $r=n-3$ for $S_n$ \cite{FLM2018}. 
For the alternating group $A_n$ it was possible to prove that the maximal rank is $\lfloor \frac{n-1}{2}\rfloor$ when $n\geq 12$ \cite{CFLM2017} and to provide examples of polytopes with highest possible rank \cite{FLM2012, FLM20122}.

The concept of a hypertope, as a generalization of a polytope, was recently introduced \cite{FLW1}. 
Some examples of highly symmetric hypertopes with a given diagram have been constructed \cite{FLW2,CFHL}.
The classification of regular hypertopes of ranks $n-1$ and $n-2$ for the symmetric group was also accomplished using a generalization of CPR graphs \cite{FL2} considering Coxeter groups with nonlinear diagrams.
In \cite{FLW2} hypertopes of rank $4$, whose rank $3$ residues are toroidal hypermaps,  called hexagonal extensions of toroidal hypermaps, were constructed.
In this work a CPR graph for the toroidal map $\{6,3\}_{(s,0)}$ ($s\geq 2$) was given. Indeed
the existence of these rank $4$ regular hypertopes is only possible if we can, in a certain way, combine the CPR graphs of its rank $3$ residues. Gh\"{u}nbaum's problem is one of the classical problems of the theory of abstract regular polytopes, still not completely solved, and it consists in the classification of locally toroidal polytopes. The problem is extensible to regular hypertopes of rank 4 with toroidal rank 3 residues and we strongly believe that  CPR graphs of regular toroidal maps may play an important role in the classification of regular hypertopes with toroidal rank 3 residues. Indeed, the group theoretical conditions for an incidence system to be a regular hypertope are much easier to check when the CPR graphs are known.

In this paper we  list all possible degrees of a faithful transitive permutation representation of the group of symmetries of a regular map of type $\{4,4\}$ and $\{3,6\}$. The results can be summarized as follows, considering $s\geq 3$:
\begin{itemize}
\item for the map  $\{4,4\}_{(s,0)}$ the possible degrees are $n=s^2$, $2ab$, $4ab$ and $8ab$ with $s=lcm(a,b)$; 
\item for the map $\{3,6\}_{(s,0)}$ the possible degrees are $n=s^2$, $2s^2$, $3ds$ (with $d|s$), $4s^2$, $6ab$ and $12ab$, where $s=lcm(a,b)$.
\item  the degrees of   $\{4,4\}_{(s,s)}$ are the degrees of $\{4,4\}_{(s,0)}$ multiplied by $2$, while the degrees of $\{3,6\}_{(s,s)}$ are the degrees of $\{3,6\}_{(s,0)}$ multiplied by $3$. 
\end{itemize}
We deal with the particular cases $s=1$ and $s=2$ separately.

This paper is the just the starting point for a study of faithful permutation representations of regular polytopes. 
Maps of other genus could also be considered in future works or other regular polytopes namely cubic toroids or locally toroidal polytopes. 

In what follows we start giving some background, mainly  on toroidal maps and CPR graphs. We will deal separately with types $\{4,4\}$ and $\{3,6\}$ but there is a common theory that will be included in a preliminary section, just after the background section. After listing all the possible degrees for each of the types, we give in each case some examples of families of CPR graphs.

%-----------------------------------------------------------------------------------------------------------------------------------------------------------------------------------------------------------------------------------------------------------------------------------------------------------
%-----------------------------------------------------------------------------------------------------------------------------------------------------------------------------------------------------------------------------------------------------------------------------------------------------------
%-----------------------------------------------------------------------------------------------------------------------------------------------------------------------------------------------------------------------------------------------------------------------------------------------------------
%-----------------------------------------------------------------------------------------------------------------------------------------------------------------------------------------------------------------------------------------------------------------------------------------------------------

\section{Toroidal regular maps}\label{back}

Consider the regular tessellations of the plane by identical squares, triangles or hexagons.  These tesselations are infinite regular 3-polytopes whose full symmetry groups are the Coxeter groups $[4, 4]$, $[3,6]$ or $[6,3]$, respectively, generated by three reflections $\rho_0,\, \rho_1,\rho_2$, as shown in Figures~\ref{map44trans} and \ref{map36trans}. Identifying opposites sides of a parallelogram, whose vertices are vertices of the tessellation, we obtain a toroidal map of one of the types $\{4,4\}$, $\{3,6\}$ or $\{6,3\}$, respectively. In what follows we will only consider toroidal maps of types $\{4,4\}$ and $\{3,6\}$. Tessellations by hexagons are dual to triangular tessellations, therefore there is a one-to-one correspondence between the degrees of maps of types $\{3,6\}$ and $\{6,3\}$. 

%-----------------------------------------------------------------------------------------------------------------------------------------------------------------------------------------------------------------------------------------------------------------------------------------------------------
%-----------------------------------------------------------------------------------------------------------------------------------------------------------------------------------------------------------------------------------------------------------------------------------------------------------

 \subsection{Toroidal Map $\{4,4\}$}
 Consider the toroidal map $\{4,4\}_{(s,t)}$ having  $V=s^2+t^2$ vertices, $2V$ edges and $V$ faces,  that is the obtained identifying opposite sides of the parallelogram with vertices $(0,0)$, $(s,t)$, $(s-t,s+t)$ and $(-s,t)$, as shown in Figure~\ref{map44trans}.
 \begin{figure}[h!]
 \begin{tiny}
 $$\xymatrix@-1.8pc{
&&&&&&&&&&&&&&&&& &&&&&&&& &&&&&&&&&\\
&&&&\ar@{-}[dddddddddddd]&&\ar@{-}[dddddddddddd]&&\ar@{-}[dddddddddddd]&&\ar@{-}[dddddddddddd]&&\ar@{-}[dddddddddddd]&&\ar@{-}[dddddddddddd]&&\ar@{-}[dddddddddddd]&*{\rho_0}\ar@{.}[dddddddddddd]&\ar@{-}[dddddddddddd]&&\ar@{-}[dddddddddddd]&&\ar@{-}[dddddddddddd]&&\ar@{-}[dddddddddddd]&&\ar@{-}[dddddddddddd]&&\ar@{-}[dddddddddddd]&&\ar@{-}[dddddddddddd]&&\ar@{-}[dddddddddddd]&&\\
&&&&&&&&&&&&&&&&& &&&&&&&& &&&&&&&&&\\
&&\ar@{-}[rrrrrrrrrrrrrrrrrrrrrrrrrrrrrrrr]&&&&&&&&&&&&&&& &&&*{\bullet}&&&&& &&&&&&&&&\\
&&&&&&&&&&&&&&&&& &&&&&&&& &&&&&&&&&\\
&&\ar@{-}[rrrrrrrrrrrrrrrrrrrrrrrrrrrrrrrr]&&&&&&&&&&&&*{\bullet}\ar@{--}[rrrrrruu]&&& &&&&&&&& &&&&&&&&&\\
&&&&&&&&&&&&&&&&& &&&&&&&& &&&&&&&&&\\
*{\rho_2}\ar@{.}[rrrrrrrrrrrrrrrrrrrrrrrrrrrrrrrrrrrr]&&\ar@{-}[rrrrrrrrrrrrrrrrrrrrrrrrrrrrrrrr]&&&&&&&&&&&&&&& &&&&&&&& &&&&&&&&&&&&&&&&&\\
&&&&&&&&&&&&&&&&& &&&&&&&& &&&&&&&&&\\
&&\ar@{-}[rrrrrrrrrrrrrrrrrrrrrrrrrrrrrrrr]&&&&&&&&&&&&&&& &&&&&*{\bullet}\ar@{--}[lluuuuuu]_(.9){(s-t,s+t)}&&&&&&&&&&&&&&&\\
&&&&&&&&&&&&&&&&& &&&&&&&& &&&&&&&&&\\
&&\ar@{-}[rrrrrrrrrrrrrrrrrrrrrrrrrrrrrrrr]&&&&&&&&&&&&&&*{\bullet}\ar@{--}[uuuuuull]^(.01){(0,0)}^(.9){(-t,s)}\ar@{--}[rrrrrruu]_(.9){(s,t)}& &&&&&&&& &&&&&&&&&\\
&&&&&&&&&&&&&&&&& &&&&&&&& &&&&&&&&&\\
&&&&&&&&&&*{\rho_1}\ar@{.}[urururururururururururur]&&&&&&&&&&&&&&&&&&&&&&&&
}$$\end{tiny}
\caption{Toroidal map of type $\{4,4\}$.}
\label{map44trans}
\end{figure}
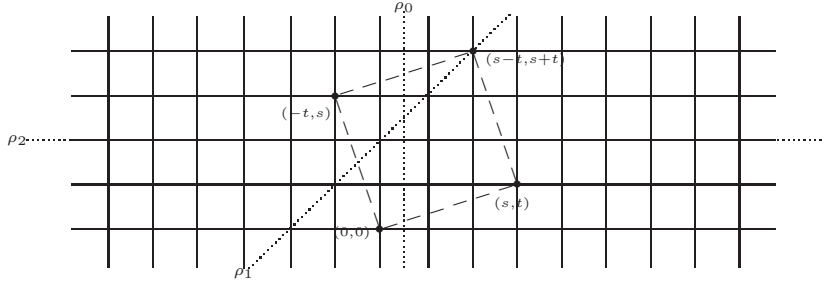
The product of two reflections $\rho_i$ and $\rho_j$ with $i, j\in\{0,1,2\}$ are still symmetries of  the identified map $\{4,4\}_{(s,t)}$ (rotational symmetries), 
but $\rho_0$, $\rho_1$ and $\rho_2$ are reflexions of $\{4,4\}_{(s,t)}$ only if $st(s-t)=0$ \cite{CM72}, that is precisely when the map is regular, 
meaning that the group of symmetries acts regularly on the set of flags (triples of mutually incident, vertex, edge and face).
Hence there are two families of toroidal regular maps, denoted by $\{4,4\}_{(s,0)}$ and $\{4,4\}_{(s,s)}$. 
Their group of symmetries are 
factorizations of the Coxeter group $[4,4]$, by 
$$(\rho_0\rho_1\rho_2\rho_1)^s=1\mbox{ or } (\rho_0\rho_1\rho_2)^{2s}=1,$$
respectively. 
%The perpendicular translations $u=\rho_0\rho_1\rho_2\rho_1$ and $v=\rho_1\rho_0\rho_1\rho_2$ generate an abelian subgroup with $u^sv^t$ being a translation of the origin $(0,0)$ to a point $(s,t)$.
The number of flags of  $\{4,4\}_{(s,0)}$ is $8s^2$ while the number of flags of $\{4,4\}_{(s,s)}$ is $16s^2$. 
Indeed the group of  $\{4,4\}_{(s,0)}$ is  isomorphic to a subgroup of index $2$ of $\{4,4\}_{(s,s)}$ and  the group of the map $\{4,4\}_{(s,s)}$ is also isomorphic to a subgroup of index $2$ of the group of the map $\{4,4\}_{(2s,0)}$. 

For the map $\{4,4\}_{(s,0)}$ consider  the unitary translations $u= \rho_0\rho_1\rho_2\rho_1$ and $v= u^{\rho_1}$. 
 
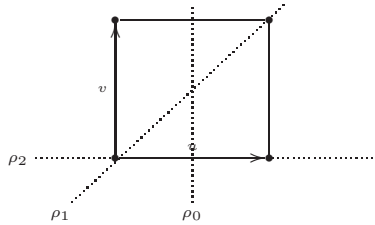
\begin{tiny}
$$\xymatrix@-1.8pc{&&&&&&&&&&&&&&\\
&&&*{\bullet}\ar@{-}[rrrrrr]&&&&&&*{\bullet}\ar@{-}[dddddd]&&&&&\\
&&&&&&&&&&&&&&\\
&&&&&&&&&&&&&&\\
&&&&&&&&&&&&&&\\
&&&&&&&&&&&&&&\\
&&&&&&&&&&&&&&\\
\rho_2\ar@{.}[rrrrrrrrrrrrrr]&&&*{\bullet}\ar@{->}[rrrrrr]^{u}\ar@{->}[uuuuuu]^{v}&&&&&&*{\bullet}&&&&&\\
&&&&&&&&&&&&&&\\
&\rho_1\ar@{.}[ururururururururur]&&&&&\rho_0\ar@{.}[uuuuuuuuu]&&&&&&&&\\
 }$$
\end{tiny} 
 We have the following equalities
\begin{equation}\label{cong1}
u^{\rho_0}=u^{-1},\,u^{\rho_2}=u,\, v^{\rho_0}=v\mbox{ and }v^{\rho_2}=v^{-1}.
\end{equation}
In the case of the map $\{4,4\}_{(s,s)}$, consider  $g:= uv=(\rho_0\rho_1\rho_2)^2$ and $h:=u^{-1}v=g^{\rho_0}$. 
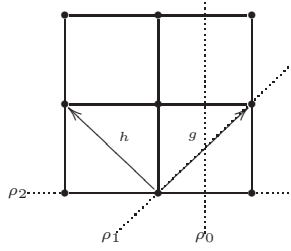
\begin{tiny}
$$\xymatrix@-1.8pc{&&&&&&&&&&&&&&&&\\
&&&&*{\bullet}\ar@{-}[dddddddd]\ar@{-}[rrrrrrrr]&&&&*{\bullet}\ar@{-}[dddddddd]&&&&*{\bullet}\ar@{-}[dddddddd]&&&&\\
&&&&&&&&&&&&&&&&\\
&&&&&&&&&&&&&&&&\\
&&&&&&&&&&&&&&&&\\
&&&&*{\bullet}\ar@{-}[rrrrrrrr]&&&&*{\bullet}&&&&*{\bullet}&&&&\\
&&&&&&&&&&&&&&&&\\
&&&&&&&&&&&&&&&&\\
&&&&&&&&&&&&&&&&\\
&&*{\rho_2}\ar@{.}[rrrrrrrrrrrr]&&*{\bullet}\ar@{-}[rrrrrrrr]&&&&*{\bullet}\ar@{->}[urururur]^g\ar@{->}[lulululu]_h&&&&*{\bullet}&&&&\\
&&&&&&&&&&&&&&&&\\
&&&&&&*{\rho_1}\ar@{.}[urururururururur]&&&&*{\rho_0}\ar@{.}[uuuuuuuuuuu]
 }$$
\end{tiny} 
We have the following equalities
\begin{equation}\label{cong2}
g^{\rho_1}=g,\,g^{\rho_2}=h^{-1}\mbox{ and }h^{\rho_1}=h^{-1}.
\end{equation}

%-----------------------------------------------------------------------------------------------------------------------------------------------------------------------------------------------------------------------------------------------------------------------------------------------------------
%-----------------------------------------------------------------------------------------------------------------------------------------------------------------------------------------------------------------------------------------------------------------------------------------------------------

\subsection{Toroidal Map $\{3,6\}$}
 Consider the map $\{3,6\}_{(s,t)}$ having $V= s^2+st+t^2$ vertices, $3V$ edges  and $2V$ faces, that is obtained identifying opposite edges of a parallelogram with vertices $(0,0)$, $(s,t)$, $(-t, s+t)$ and $(s-t, s+2t)$, as shown in Figure~\ref{map36trans}.
\begin{figure}[h!]
\begin{tiny}
$$\xymatrix@-1.8pc{
&&&&&&&&&&&&&&&&& &&&&&&&& &&&&&&&&&\\ 
&&&&&&&&&&&&&&&&&&&&&&&&& &&&&&&&&&\\
&& *{} \ar@{-}[ddddddddddrrrrrrrrrr] &&&&*{} \ar@{-}[ddddddddddrrrrrrrrrr]&&&&*{} \ar@{-}[ddddddddddrrrrrrrrrr]&&&&*{} \ar@{-}[ddddddddddrrrrrrrrrr]&&&&*{} \ar@{-}[ddddddddddrrrrrrrrrr]&&&&*{} \ar@{-}[ddddddddddrrrrrrrrrr]&&& &\ar@{-}[ddddddrrrrrr]&&&&&& &&\\
&& *{} \ar@{-}[rrrrrrrrrrrrrrrrrrrrrrrrrrrrrr] &&&&&&&&&&&&&&&&&&&&&*{\bullet}\ar@{--}[ddllllllllll]_(.01){(s-t,s+2t)} \ar@{--}[ddddddll]&& &&&&&&&&&\\
&& &&&&&&&&&&&&&&&&&&&&&&& &&&&&&&&&\\
&&*{} \ar@{-}[rrrrrrrrrrrrrrrrrrrrrrrrrrrrrr] &&&&&&&&&&&*{\bullet}&&&&&&&&&&&& &&&&&&&&&\\
&&\ar@{-}[ddddddrrrrrr] &&&&&&&&&&&&&&&&&&&&&&& &&&&&&&&&\\
 && *{} \ar@{-}[rrrrrrrrrrrrrrrrrrrrrrrrrrrrrr]&&&&&&&&&&&&&&&&&&&&&&& &&&&&&&&&\\
&&  \ar@{-}[uuuuuurrrrrr]&&&&&&&&&&&&&&&&&&&&&&& &&&&&&&&&\\
*{\rho_2} \ar@{.}[rrrrrrrrrrrrrrrrrrrrrrrrrrrrrrrrrr] && *{} \ar@{-}[rrrrrrrrrrrrrrrrrrrrrrrrrrrrrr] &&& &&&&&&&&&&&&&&&&*{\bullet}&&&& &&&&&&&&&\\
 &&   &&&&&&&&&&&&&&&&&&&&&&& &&&&&&&&&\\
  &&   *{} \ar@{-}[rrrrrrrrrrrrrrrrrrrrrrrrrrrrrr] &&&&&&&&&*{\bullet}\ar@{--}[uuuuuurr]_(.01){(0,0)}^(.9){(-t,s+t)}\ar@{--}[uurrrrrrrrrr]_(.9){(s,t)}&&&&&&&&&&&&&& &&&&&&&&&\\
 &&*{} \ar@{-}[uuuuuuuuuurrrrrrrrrr] &&*{\rho_1} \ar@{.}[uuuuuuuuurrrrrrrrrrrrrrrrrrrrrrrrrrr]&&*{} \ar@{-}[uuuuuuuuuurrrrrrrrrr]&&&&*{} \ar@{-}[uuuuuuuuuurrrrrrrrrr]&&&&*{} \ar@{-}[uuuuuuuuuurrrrrrrrrr]&&&&*{} \ar@{-}[uuuuuuuuuurrrrrrrrrr]&&&&*{} \ar@{-}[uuuuuuuuuurrrrrrrrrr]&& &&\ar@{-}[uuuuuurrrrrr]&&&&&&&&\\
&&   &&&&&&&&&&&&&*{\rho_0} \ar@{.}[uuuuuuuuuuuuu]&& &&&&&&&& &&&&&&&&&\\
 }$$
\end{tiny}
\caption{Toroidal map of type $\{3,6\}$.}
\label{map36trans}
\end{figure}
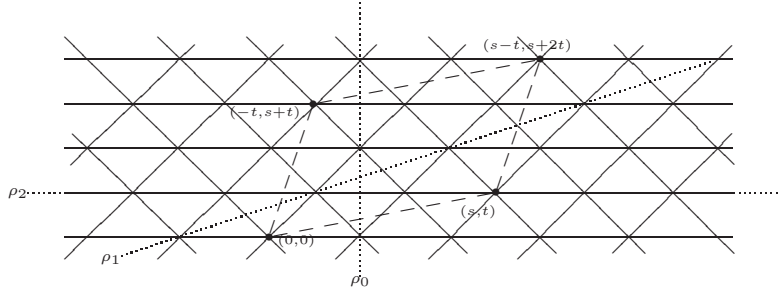
The involutions $\rho_0$, $\rho_1$ and $\rho_2$ are symmetries of  $\{3, 6\}_{(s,t)}$ only if $st(s-t)=0$, that is the condition for regularity of $\{3, 6\}_{(s,t)}$. 
Hence we distinguish two families of toroidal regular maps of type $\{3,6\}$:  $\{3, 6\}_{(s,0)}$ and $\{3, 6\}_{(s,s)}$. 
The group of symmetries of  $\{3, 6\}_{(s,0)}$ and  $\{3, 6\}_{(s,s)}$ are factorizations of  the Coxeter group  $[3, 6]$ by 
$$(\rho_0\rho_1\rho_2)^{2s}=1\mbox{ and }((\rho_2\rho_1)^2\rho_0)^{2s}=1,$$
 respectively.
The number of flags of $\{3, 6\}_{(s,0)}$ is $12s^2$ while the number of flags of $\{3, 6\}_{(s,s)}$ is $36s^2$.
Indeed the group of $\{3, 6\}_{(s,0)}$ is isomorphic to a a subgroup of index $3$ of the group $\{3, 6\}_{(s,s)}$ and the group of  $\{3,6\}_{(s,s)}$ is also isomorphic to a subgroup of index $3$ of the group of the map $\{3,6\}_{(3s,0)}$. For the map $\{3,6\}_{(s,0)}$ consider  the unitary translations $u=\rho_0(\rho_1\rho_2)^2\rho_1$, $v= u^{\rho_1}=(\rho_0\rho_1\rho_2)^2$ and $t=u^{-1}v$. 
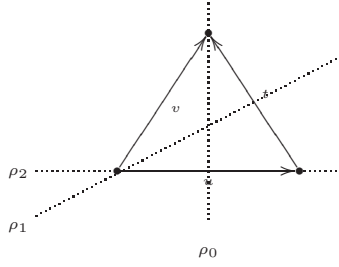
\begin{tiny}
$$\xymatrix@-1.8pc{&&&&&&&&&&&&&&\\
&&&&&&&&& &&&&&\\
&&&&&&&&*{\bullet}&&&&&&\\
&&&&&&&&&&&&&&\\
&&&&&&&&&&&&&&\\
&&&&&&&&&&&&&&\\
&&&&&&&&&&&&&&\\
&&&&&&&&&&&&&&\\
\rho_2\ar@{.}[rrrrrrrrrrrrrr]&&&&*{\bullet} \ar@{->}[rrrrrrrr]_u\ar@{->}[uuuuuurrrr]_v&&&&&&&&*{\bullet}\ar@{->}[lllluuuuuu]_t &&\\
&&&&&&&&&&&&&&\\
\rho_1\ar@{.}[uuuuuuurrrrrrrrrrrrrr]&&&&&&&&\ar@{.}[uuuuuuuuuu]&&&&&&\\
&&&&&&&&*{\rho_0}&&&&&&\\
 }$$
\end{tiny}
We have the following equalities
\begin{equation}\label{cong3}
u^{\rho_0}=u^{-1},\,u^{\rho_2}=u,\, v^{\rho_0}=t\mbox{ and }v^{\rho_2}=t^{-1}.
\end{equation}

In the case of the map $\{3,6\}_{(s,s)}$, consider  $g:= uv=(\rho_0(\rho_1\rho_2)^2)^2$,  $h:=u^{-2}v=g^{\rho_0}$ and $j:=hg$. We have the following equalities
\begin{equation}\label{cong4}
g^{\rho_1}=g,\,g^{\rho_2}=h^{-1}\mbox{ and }h^{\rho_1}=j^{-1}.
\end{equation}
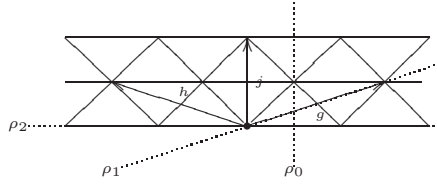
\begin{tiny}
$$\xymatrix@-1.8pc{
&&&&&& &&&& &&&& &&&&&\\
&&&&&& &&&& &&&& &&&&&\\
&&*{}\ar@{-}[rrrrrrrrrrrrrrrr]\ar@{-}[ddddrrrr]&&&&*{}\ar@{-}[ddddrrrr]\ar@{-}[ddddllll]&&&&*{}\ar@{-}[ddddllll]\ar@{-}[ddddrrrr]&&&&*{}\ar@{-}[ddddllll]\ar@{-}[ddddrrrr]&&&&*{}\ar@{-}[ddddllll]&\\
&&&&&& &&&& &&&& &&&&&\\
&&*{}\ar@{-}[rrrrrrrrrrrrrrrr]&&*{}&&&&*{}&&&&*{}&&&&*{}\ar@{.}[urrr] &&&\\
&&&&&& &&&& &&&& &&&&&\\
*{\rho_2}*{}\ar@{.}[rrrrrrrrrrrrrrrrrrr]&&*{}\ar@{-}[rrrrrrrrrrrrrrrr]&&&&*{} &&&& *{\bullet}\ar@{->}[uurrrrrr]_g\ar@{->}[uullllll]_h\ar@{->}[uuuu]_j&&&& *{}&&&&*{}&\\
&&&&&& &&&& &&&& &&&&&\\
&&&&*{\rho_1}\ar@{.}[uuuurrrrrrrrrrrr]&&&&&&&&*{\rho_0}\ar@{.}[uuuuuuuu] && &&&&&\\
 }$$
 \end{tiny}
%-----------------------------------------------------------------------------------------------------------------------------------------------------------------------------------------------------------------------------------------------------------------------------------------------------------
%-----------------------------------------------------------------------------------------------------------------------------------------------------------------------------------------------------------------------------------------------------------------------------------------------------------

 \subsection{String C-groups and CPR graphs}

The group of symmetries of a toroidal map is a string C-group of rank 3, the unique exception is the group of $\{4,4\}_{(1,0)}$ (that is not a polytope).
In general a string C-group of rank $r$ is a group generated by $r$ involutions
$\rho_0, \rho_1,\ldots , \rho_{r-1}$ satisfying the following conditions.
\begin{itemize}
\item $(\rho_i\rho_j)^2=1$ if $|i-j|>1$ for every $i,\,j\in I:=\{0,\ldots,r-1\}$;
\item $\langle \rho_j\,|\,j\in J\rangle\cap\langle \rho_k\,|\,k\in K\rangle=\langle \rho_i\,|\,i\in J\cap K\rangle$ for any pair, $J$ and $K$, of subsets of $I$.
\end{itemize}
If $G$ is a string C-group of rank $r$ and has degree $n$  then it can be represented by a graph with $n$ vertices and with an edge $\{a,b\}$ with label $i$ whenever $a\rho_i=b$. This graph is called a \emph{CPR graph of degree $n$} \cite{CPR}.
Let $\mathcal{G}$ be a CPR graph and let $\mathcal{G}_J$ denote the graph with the same set of vertices of $\mathcal{G}$ and with the set of edges having labels in $J\subseteq \{0,\ldots, r-1\}$.
A CPR graph satisfies the following properties:
\begin{itemize}
\item $\mathcal{G}_{\{i\}}$ is a matching;
\item if $i$ and $j$ are nonconsecutive the connected components of $\mathcal{G}_{\{i,j\}}$ are either single edges, double edges or squares with alternating labels (alternating $\{i,j\}$-squares).
\end{itemize}
A graph satisfying these properties gives a faithful permutation representation of a string group $G$ generated by involutions. If in addition the intersection condition is satisfied, the graph is a CPR graph of an abstract regular polytope. If the graph is connected, the degree is the index of the stabilizer of a point. There is a one-to-one correspondence between abstract regular polytopes and string C-groups. When a regular polytope is described by a certain CPR graph $\mathcal{G}$, we say that the number of vertices of $\mathcal{G}$ is the degree of the regular polytope.
%-----------------------------------------------------------------------------------------------------------------------------------------------------------------------------------------------------------------------------------------------------------------------------------------------------------
%-----------------------------------------------------------------------------------------------------------------------------------------------------------------------------------------------------------------------------------------------------------------------------------------------------------
%-----------------------------------------------------------------------------------------------------------------------------------------------------------------------------------------------------------------------------------------------------------------------------------------------------------
%-----------------------------------------------------------------------------------------------------------------------------------------------------------------------------------------------------------------------------------------------------------------------------------------------------------

\section{Preliminary results}
  
In this section we include the results that can be applied to both the toroidal regular maps of type $\{4,4\}$ and of type $\{3,6\}$.
We start by proving a proposition on the degrees of a transitive action of a direct product of cyclic groups for later use.

\begin{proposition}\label{cxc}
Let $H=C_a\times C_b$ be a direct product of cyclic groups of orders $a$ and $b$ resp. with $a\geq b$. If $H$  acts transitively on $d$ points then either $d=a$ and $b|a$ ($b<a$), or $d=ab$.
\end{proposition}
\begin{proof}
Let $C_a=\langle \alpha \rangle$ and $C_b=\langle \beta \rangle$ with $\alpha$ and $\beta$ being permutations on $d$ points.
First $\alpha$ and $\beta$ cannot be both transitive, otherwise one is a power of the other, a contradiction.
Thus either $\alpha$ or $\beta$ is intransitive. If $\alpha$ is transitive and $\beta$ is intransitive, the orbits of $\beta$ form a block system for $H$. 
More precisely, $d=a$ and $\beta$ is a product of disjoint cycles of size $b$ with $b|a$.
As $a\geq b$ it is not possible to have $\beta$ transitive and $\alpha$ intransitive. 

Now assume that $\alpha$ and $\beta$ are both intransitive.
Then there are two block systems for $H$, one block system whose blocks are the $\langle \alpha\rangle $-orbits and another whose blocks are the $\langle \beta \rangle $-orbits.
Let $f:H\mapsto S_{d/a}$ be homomorphism induced by the the action $H$ on the $\langle \alpha\rangle $-orbits. We have that $\langle \alpha \rangle\leq Ker(f)$. 
If $\beta$ swaps a pair of vertices inside a $\langle \alpha\rangle $-orbit, then $\beta$ fixes the entire $\langle \alpha\rangle $-orbit, and therefore $H$ is intransitive, a contradiction.
Indeed as $H$ is transitive $f(\beta)$ must be a cycle of order $b$. Moreover $\beta$ is fixed-point-free, thus  $\beta$ is a product of $a$ cycles of size $b$. Hence $d=ab$.
\end{proof}

In what follows, let $G$ be the group of symmetries of degree $n$ of a toroidal regular map of type $\{4,4\}$ or $\{3,6\}$ and let $T$ be a translation group as follows:  $T:=\langle u,v\rangle$ for the maps $\{4,4\}_{(s,0)}$ and $\{3,6\}_{(s,0)}$ and; $T:=\langle g,h\rangle$ for the maps $\{4,4\}_{(s,0)}$ and $\{3,6\}_{(s,0)}$  (where $u$, $v$, $g$ and $h$  are as in Section~\ref{back}). 

%--------------------------------------------
\begin{lemma}\label{Ttran}
If $T$ is transitive, then $n=s^2$. In this case $T$ is regular and $G\cong T\rtimes G_1$ where $G_1$ is the stabilizer of the identity.
\end{lemma}
\begin{proof}
We have that $T$ is a direct product of cyclic groups of order $s$. 
Thus by Proposition~\ref{cxc} $n=s^2$.
Moreover the action of $T$ is regular.
Hence $G\cong T\rtimes G_1$ is $G_1$ is the stabilizer of the identity (Section 1.7, \cite{C99}).
\end{proof}

%-------------------------------------

\begin{proposition}\label{Tintss}
If $G$ is the group of $\{4,4\}_{(s,s)}$ or $\{3,6\}_{(s,s)}$, then $T$ is intransitive.
\end{proposition}
\begin{proof} First suppose that $G$ is the group of the map $\{4,4\}_{(s,s)}$  and that $T$ is transitive. 
Let $\alpha=\rho_0\rho_1\rho_2\rho_1$.
As $T$ is  regular, $1\alpha=g^ah^b$.
As $\alpha$ commutes with both $g$ and $h$, $g^ih^j\alpha=g^{i+a}h^{j+b}$ and $g^ih^j\alpha^s=g^{i+sa}h^{j+sb}=g^ih^j$. Hence the order of $\alpha$ is at most $s$, a contradiction.

Analogously, for the map  $\{3,6\}_{(s,s)}$, if $T$ is transitive then  the order of $\alpha=(\rho_0\rho_1\rho_2)^2$ is at most $s$, giving a contradiction.

%Now let $G$ be the group of the map $\{3,6\}_{(s,s)}$  and suppose that $T$ is transitive. 
%Let $\alpha=(\rho_0\rho_1\rho_2)^2$.
%As $T$ is  regular, $1\alpha=g^ah^b$.
%As $\alpha$ commutes with both $g$ and $h$, $g^ih^j\alpha=g^{i+a}h^{j+b}$ and $g^ih^j\alpha^s=g^{i+sa}h^{j+sb}=g^ih^j$. Hence the order of $\alpha$ is at most $s$, a contradiction.
 \end{proof}
 
%-----------------------------------------------------------------------------------------------------------------------------------------------------------------------------------------------------------------------------------------------------------------------------------------------------------
%-----------------------------------------------------------------------------------------------------------------------------------------------------------------------------------------------------------------------------------------------------------------------------------------------------------
%-----------------------------------------------------------------------------------------------------------------------------------------------------------------------------------------------------------------------------------------------------------------------------------------------------------
%-----------------------------------------------------------------------------------------------------------------------------------------------------------------------------------------------------------------------------------------------------------------------------------------------------------

 \begin{lemma}\label{Gimp}
If $n\neq s^2$ then $G$ is embedded into $S_k\wr S_m$  with $n=km$ $(m,\,k>1)$ and 
\begin{enumerate}
\item[(i)] $k=ab$ where $s=lcm(a,b)$ and,
\item[(ii)] $m$ is a divisor of $\frac{|G|}{s^2}$. 
 
\end{enumerate}
\end{lemma}
\begin{proof}
By Lemma~\ref{Ttran} $T$ is intransitive, thus $G$ is embedded into $S_k\wr S_m$, with $k$ being the size of an orbit of $T$  and $n=km$. 

(i) Consider that $\alpha$ and $\beta$ are the actions of the generators of $T$ on a block $B$. Suppose that the orders of $\alpha$ and $\beta$ are $a$ and $b$, respectively.
The group $\langle \alpha \rangle\times \langle \beta \rangle$ acts transitively on $B$. Assume that $a\geq b$, then by Lemma~\ref{cxc} either $k=a$ and $b|a$, or $k=ab$.
By the relations~(\ref{cong1}), (\ref{cong2}), (\ref{cong3}) and (\ref{cong4}) of Section \ref{back}, the order of the actions of the generators of $T$ on any other block is either $a$ or $b$. Hence $s=lcm(a,b)$.

(ii) Now consider the induced action of $G$ on the set of $m$ blocks and the induced homomorphim $f:G\to S_m$. 
The kernel of this homomorphism has size at least $s^2$, as it contains $T$. Hence the size of $Im(f)$ is a divisor $\frac{|G|}{s^2}$.

\end{proof}

%-----------------------------------------------------------------------------------------------------------------------------------------------------------------------------------------------------------------------------------------------------------------------------------------------------------
%-----------------------------------------------------------------------------------------------------------------------------------------------------------------------------------------------------------------------------------------------------------------------------------------------------------

It is well known that a group acts transitively on a set of cosets of any subgroup and that this action is faithful if and only if the subgroup is core-free.
 Conversely, having a group acting faithfully on a set, the stabilizer of a point for that action is core-free. Thus there is a correspondence between core-free subgroups and faithful actions.
 Moreover, if $G$ has a faithful transitive permutation representation of degree $n$ and is a subgroup of index $\alpha$ of $K$, then $K$ has a faithful permutation representation of degree $\alpha n$. Consequently, we have the following corollary.

\begin{comment}
\begin{proof}
 Let $H$ be the core-free subgroup of $G$ with index $n$, that is, $\bigcap_{g\in G} H^g =\{1 \}$.
 As $G$ is a subgroup $K$ then $\bigcap_{g\in K} H^g =\{ 1 \}$, meaning that $H$ is also a core-free subgroup of $K$ with index $\alpha n$.
 Therefore, there is a faithful transitive permutation representation of $K$ with degree $\alpha n$.
\end{proof}
\end{comment}

\begin{corollary}\label{degreesss}
 If $n$ is a degree of $\{4,4\}_{(s,0)}$ (resp. $\{4,4\}_{(s,s)}$) then $2n$ is a degree of $\{4,4\}_{(s,s)}$ (resp. $\{4,4\}_{(2s,0)}$); and
if $n$ is the degree of $\{3,6\}_{(s,0)}$ (resp. $\{3,6\}_{(s,s)}$) then $3n$ is a degree of  $\{3,6\}_{(s,s)}$ (resp. $\{3,6\}_{(3s,0)}$);
\end{corollary}

%-----------------------------------------------------------------------------------------------------------------------------------------------------------------------------------------------------------------------------------------------------------------------------------------------------------
%-----------------------------------------------------------------------------------------------------------------------------------------------------------------------------------------------------------------------------------------------------------------------------------------------------------
%-----------------------------------------------------------------------------------------------------------------------------------------------------------------------------------------------------------------------------------------------------------------------------------------------------------
%-----------------------------------------------------------------------------------------------------------------------------------------------------------------------------------------------------------------------------------------------------------------------------------------------------------

\section{The maps of type $\{4,4\}$}

In this section we consider the regular maps of type $\{4,4\}$, determine all possible degrees for these maps and give CPR graphs for some of those degrees.

\subsection{The possible degrees for the map  $\{4,4\}_{(s,0)}$}

The groups of $\{4,4\}_{(s,0)}$ ($s>2$) act faithfully on the sets of vertices, faces,  edges, darts and flags (as the dihedral groups $\langle \rho_i, \rho_j\rangle$ and its subgroups are core-free, with $i,j\in\{0,1,2\}$). Let us consider the exceptional cases $s\in\{1,2\}$. 
The only proper core-free subgroups of the group of the map $\{4,4\}_{(1,0)}$ have order two, thus the possible degrees for the map $\{4,4\}_{(1,0)}$ are $4$ and $8$. 
For the group of the map  $\{4,4\}_{(2,0)}$ we have that $\langle \rho_1,\rho_2\rangle$ and $\langle \rho_0,\rho_1\rangle$ have nontrivial core, that is $\langle \rho_2^{\rho_1}\rangle$ and $\langle \rho_0^{\rho_1}\rangle$, respectively. Thus the map $\{4,4\}_{(2,0)}$ is an example of a map whose actions on the vertices and faces are non-faithful. But $\langle \rho_0,\rho_2\rangle$, $\langle \rho_0\rangle$ and $G$ itself, are core-free, therefore $8$, $16$ and $32$ are the possible degrees for  $\{4,4\}_{(2,0)}$. 
In what follows we determine the other possible degrees for the maps $\{4,4\}_{(s,0)}$ ($s>2$). 

%--------------------------------------------------------------------------------------------------------------------------------------------------------------------------------------------------------------------------------------

\begin{proposition}\label{core}
Let $G$ be the group of $\{4,4\}_{(s,0)}$ ($s>2$). If $a$ and $b$ are nonnegative integers and  $s=lcm(a,b)$ then 
\begin{enumerate}
\item $H=\langle u^a, v^b\rangle$ is core-free and $|G:H|=8ab$,
\item $H=\langle u^a, v^b\rangle\rtimes \langle \rho_0 \rangle$ is core-free and $|G:H|=4ab$,
\item if $ab\neq s$ then $H=\langle u^a, v^b\rangle\rtimes \langle \rho_0,\rho_2\rangle$ is core-free and $|G:H|=2ab$, and
\item $H=\langle u \rangle\rtimes \langle \rho_0,\rho_2\rangle$ is core-free and $|G:H|=2s$.
\end{enumerate}
\end{proposition}
\begin{proof}
(1) We have that $H\cap H^{\rho_1}=\langle u^a, v^b\rangle\cap \langle u^b, v^a\rangle$ is trivial.
Hence $H$ is core-free. The order of $H$ is $\frac{s^2}{ab}$ thus $|G:H|=8ab$.

(2) Suppose that $x\in H\cap H^{\rho_1}=\langle u^a, v^b\rangle\rtimes \langle \rho_0\rangle\cap \langle u^b, v^a\rangle\rtimes \langle \rho_0^{\rho_1}\rangle$.
If $x\notin T$ then $\rho_0 \rho_0^{\rho_1}\in T$, a contradiction. Thus  $x\in T$ and therefore as in (1) we conclude that $x=1$.
The order of $H$ is $\frac{2s^2}{ab}$ thus $|G:H|=4ab$.

(3) Suppose that $x\in H\cap H^{\rho_1}=\langle u^a, v^b\rangle\rtimes \langle \rho_0,\rho_2\rangle\cap \langle u^b, v^a\rangle\rtimes \langle \rho_0^{\rho_1},\rho_2^{\rho_1}\rangle$.
If $x\notin T$ then $\rho_0^i\rho_2^j (\rho_0^{\rho_1})^k (\rho_2^{\rho_1})^l\in K:=\langle u^a, v^b,u^b, v^a\rangle$ for some $i,\,j,\,k,\,l\in\{0,1\}$. 
This is only possible for $(i,\,j,\,k,\,l)\in\{(0,1,1,0),(1,1,1,1),(1,0,0,1)\}$. 
Then we get either $v\in K$,  $u\in K$ or $v^{-1}u\in K$ which is possible only if $s=ab$, a contradiction.
Thus  $x\in T$ and therefore, as in  (1), we conclude that $x=1$.
The order of  $H$ is $\frac{4s^2}{ab}$ thus $|G:H|=2ab$.

(4)  We have that $\langle u, \rho_0,\rho_2\rangle\cap \langle v, \rho_0^{\rho_1},\rho_2^{\rho_1}\rangle=\langle \rho_2,\rho_2^{\rho_1}\rangle$ that is core-free.
\end{proof}

\begin{theorem}\label{T1}
Let $s>2$. There exists a CPR graph of a toroidal map $\{4,4\}_{(s,0)}$  with $n$ vertices if and only if $n\in\{s^2, 2ab,\,4ab,\,8ab\}$ with $s=lcm(a,b)$. 
\end{theorem}
\begin{proof}
The number of vertices of $\{4,4\}_{(s,0)}$  is $s^2$, hence, as observed before, $\{4,4\}_{(s,0)}$ has a faithful permutation representation of degree $s^2$, corresponding to the core-free subgroup  $\langle \rho_1,\rho_2\rangle$.  The rest follows from Lemma~\ref{Ttran} and ~\ref{Gimp} and Proposition \ref{core}.
\end{proof}

%-----------------------------------------------------------------------------------------------------------------------------------------------------------------------------------------------------------------------------------------------------------------------------------------------------------
%-----------------------------------------------------------------------------------------------------------------------------------------------------------------------------------------------------------------------------------------------------------------------------------------------------------

\subsection{The possible degrees for the map  $\{4,4\}_{(s,s)}$}

The group of  $\{4,4\}_{(1,1)}$ is isomorphic to a subgroup of $\{4,4\}_{(2,0)}$, in addition it is not possible to have a CPR graph for  $\{4,4\}_{(1,1)}$ with only $4$ points, hence the possible degrees for $\{4,4\}_{(1,1)}$ are $8$ and $16$. Let us now consider $s>1$. By Corollary~\ref{degreesss} and Theorem~\ref{T1} there are faithful permutation representations for $\{4,4\}_{(s,s)}$ for $n\in\{2s^2,4ab,8ab,16ab\}$ with $s=lcm(a,b)$. In what follows we prove that those are the unique possibilities for $n$. By Proposition~\ref{Tintss} and Lemma~\ref{Gimp} the possibility to rule out is $n=2ab$ with $s=lcm(a,b)$ and $n\neq 2s^2$. Moreover in that case $G$ is embedded into $S_k\wr S_m$ where the blocks are the orbits of $T=\langle g,h\rangle$. That is the case we need to deal with.

\begin{lemma}\label{ssm=2}
If $m=2$ then $n=2s^2$.
\end{lemma}
\begin{proof}
Let $G$ be the group of $\{4,4\}_{(s,s)}$ acting faithfully on $n$ points.
Suppose that $T=\langle g,h\rangle$ has two orbits. 
Consider the group $K$, isomorphic to the group of symmetries of $\{4,4\}_{(2s,0)}$, containing $G$. 
We have that $K$ acts faithfully on two copies of the set of $n$ points. 
Let $H=\langle u,v\rangle<K$ be the translation group for $\{4,4\}_{(2s,0)}$, as in Section~\ref{back}. We have that $|u|=|v|=2s$. Moreover, if $x$ is a point on one of the copies, $x(uv)^s$ is on the other copy. In addition 
$T$ is a proper subgroup of $H$, thus $H$ must be transitive on $2n$ points and therefore it acts regularly on $2n$ points. Hence, $H$ has order $(2s)^2$,  $2n=(2s)^2$, as wanted.
\end{proof}
 
%-----------------------------------------------------------------------------------------------------------------------------
\begin{theorem}\label{T2}
Let $s>1$. There exists a CPR graph of a toroidal map $\{4,4\}_{(s,s)}$ with $n$ vertices if and only if $n\in\{2s^2,4ab,8ab,16ab\}$ with $s=lcm(a,b)$. \end{theorem}
\begin{proof}
This result follows from Proposition~\ref{Tintss}, Lemma~\ref{Gimp},  Corollary~\ref{degreesss}, Theorem~\ref{T1} and Lemma~\ref{ssm=2}. 
\end{proof}

%----------------------------------------------------------------------------------------------------------------------------------------------------------------------------------------------------------------------------------------------------------------
%----------------------------------------------------------------------------------------------------------------------------------------------------------------------------------------------------------------------------------------------------------------
\subsection{CPR graphs of $\{4,4\}_{(s,0)}$}\label{CPR44}

\begin{proposition}\label{s02s}
The following graphs are CPR graphs of $\{4,4\}_{(s,0)}$ of degree $2s$ ($s\geq 3$).

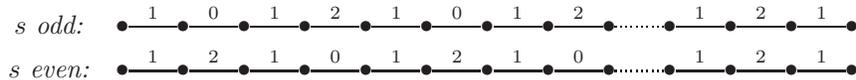
\begin{tabular}{cc}
$s$ odd: & $\xymatrix@-1.5pc{
*{\bullet}  \ar@{-}[rr]^1&& *{\bullet} \ar@{-}[rr]^0&&*{\bullet}  \ar@{-}[rr]^1&& *{\bullet} \ar@{-}[rr]^2&&*{\bullet}  \ar@{-}[rr]^1&& *{\bullet} \ar@{-}[rr]^0&&*{\bullet}  \ar@{-}[rr]^1&& *{\bullet} \ar@{-}[rr]^2&& *{\bullet} \ar@{.}[rr]&& *{\bullet} \ar@{-}[rr]^1&& *{\bullet} \ar@{-}[rr]^2&& *{\bullet} \ar@{-}[rr]^1&& *{\bullet} } $\\
$s$ even: & $\xymatrix@-1.5pc{
*{\bullet}  \ar@{-}[rr]^1&& *{\bullet} \ar@{-}[rr]^2&& *{\bullet}  \ar@{-}[rr]^1&& *{\bullet} \ar@{-}[rr]^0&&*{\bullet}  \ar@{-}[rr]^1&& *{\bullet} \ar@{-}[rr]^2&& *{\bullet}  \ar@{-}[rr]^1&& *{\bullet} \ar@{-}[rr]^0&&*{\bullet}  \ar@{.}[rr] && *{\bullet} \ar@{-}[rr]^1&& *{\bullet} \ar@{-}[rr]^2&& *{\bullet} \ar@{-}[rr]^1&& *{\bullet} } $\\
\end{tabular}

 \noindent Moreover the stabilizer of a point is, up to a conjugacy, $\langle u\rangle\rtimes\langle \rho_0,\rho_2\rangle$.
\end{proposition}
\begin{proof}
Let $G=\langle \rho_0,\rho_1,\rho_2\rangle$ be the group with one of the given permutation representation graphs. Let $x$ be the vertex of degree 1 on the left. Note that, when $s$ is odd, $\rho_0\rho_1\rho_2\rho_1$ fixes $x$, fixes all the $s-1$ vertices that are swapped by $\rho_2$ and cyclicly permutes the remaining vertices. When $s$ is even, $\rho_0\rho_1\rho_2\rho_1$ fixes all the $s$ vertices that are swapped by $\rho_2$ and cyclicly permutes the remaining vertices.  This shows that, in both cases, $(\rho_0\rho_1\rho_2\rho_1)^s=1$. In addition $\rho_0^2=\rho_1^2=\rho_2^2=(\rho_0\rho_1)^4=(\rho_1\rho_2)^4=(\rho_0\rho_1)^2=1$. 
Hence $G$ a subgroup of the group of symmetries of $\{4,4\}_{(s,0)}$. Let us prove that $|G|=8s^2$. 

First consider the case $s$ odd.  The stabilizer $G_x$ of $x$ contains $\rho_0$, $\rho_2$ and  $u:=\rho_0\rho_1\rho_2\rho_1$. Hence $G_x$ contains $\langle \rho_0,\rho_2, u\rangle \cong \langle \rho_2\rangle \times D_{2s}$, thus  $|G_x|\geq 4s$ and therefore, $|G|\geq 8s^2$.

Now let $s$ be even. The stabilizer $G_x$ of $x$ contains $\rho_0$, $\rho_2$ and  $v:=\rho_1\rho_0\rho_1\rho_2=u^{\rho_1}$. As before $|G|\geq 8s^2$.

This concludes the proof that the given graphs are CPR graphs of $\{4,4\}_{(s,0)}$.
\end{proof}

%-------------------------------------

\begin{proposition}\label{4s1}
The following graph is a CPR graph of  $\{4,4\}_{(s,0)}$ of degree $4s$ ($s\geq 2$).
$$
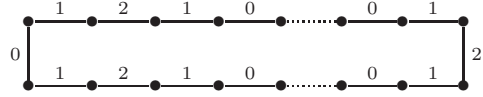
\xymatrix@-1.5pc{
*{\bullet}  \ar@{-}[dd]_0\ar@{-}[rr]^1&& *{\bullet} \ar@{-}[rr]^2&& *{\bullet} \ar@{-}[rr]^1&& *{\bullet} \ar@{-}[rr]^0 &&*{\bullet} \ar@{.}[rr]&& *{\bullet} \ar@{-}[rr]^0&& *{\bullet} \ar@{-}[rr]^1&& *{\bullet} \ar@{-}[dd]^2\\
&&&&&&\\
*{\bullet}   \ar@{-}[rr]^1&& *{\bullet} \ar@{-}[rr]^2 && *{\bullet} \ar@{-}[rr]^1&& *{\bullet} \ar@{-}[rr]^0 && *{\bullet}\ar@{.}[rr]&& *{\bullet} \ar@{-}[rr]^0 &&*{\bullet} \ar@{-}[rr]^1&& *{\bullet} }$$

 \noindent Moreover the stabilizer of a point is, up to a conjugacy, $\langle u\rangle\rtimes\langle \rho_0\rangle$.

\end{proposition}
\begin{proof}
Let $G=\langle \rho_0,\rho_1,\rho_2\rangle$ be the group with the given CPR
graph. First we have $\rho_0^2=\rho_1^2=\rho_2^2=(\rho_0\rho_1)^4=(\rho_1\rho_2)^4=(\rho_0\rho_2)^2=(\rho_0\rho_1\rho_2\rho_1)^s=1$.
Hence $G$ is a subgroup of the group of symmetries of $\{4,4\}_{(s,0)}$, particularly $|G|\leq 8s^2$.
Now consider a vertex $x$ not fixed by $\rho_2$. 
The stabilizer $G_x$ of $x$ contains $\rho_0$ and $u$. Thus $G_x$ contains $\langle \rho_0,u\rangle$ that is a group of order $2s$. Thus
$|G_x|\geq 2s$, which proves that $|G|=8s^2$. Consequently the graph is a CPR graph of $\{4, 4\}_{(s,0)}$.
\end{proof}
%-------------------------------------

\begin{proposition}\label{4s2}
The following graph is a CPR graph of  $\{4,4\}_{(s,0)}$ of degree $4s$ ($s\geq 2$).
$$
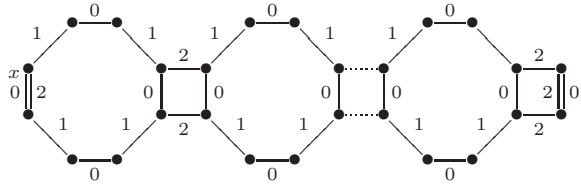
\xymatrix@-1.8pc{
&&*{\bullet} \ar@{-}[ddll]_1 \ar@{-}[rr]^0&& *{\bullet} \ar@{-}[ddrr]^1&& &&  && *{\bullet}  \ar@{-}[rr]^0\ar@{-}[ddll]_1 &&*{\bullet}\ar@{-}[ddrr]^1 && && && *{\bullet}  \ar@{-}[ddll]_1 \ar@{-}[rr]^0&&*{\bullet}\ar@{-}[ddrr]1 &&  
&& && \\
&&&&&&&&&&&&&&&&&&&&&\\
*{\bullet}\ar@{=}[dd]_0^2_(.1)x && && &&*{\bullet}  \ar@{-}[rr]^2\ar@{-}[dd]_0&&  *{\bullet}\ar@{-}[dd]^0 &&  && && *{\bullet}  \ar@{.}[rr]\ar@{-}[dd]_0 && *{\bullet}\ar@{-}[dd]^0  && && && *{\bullet}   \ar@{-}[rr]^2\ar@{-}[dd]_0&&*{\bullet}  \ar@{=}[dd]^0_2 \\
&&&&&&&&&&&&&&&&&&&&&&\\ 
*{\bullet}  \ar@{-}[ddrr]^1&& && &&*{\bullet}  \ar@{-}[rr]_2\ar@{-}[ddll]_1&&  *{\bullet} \ar@{-}[ddrr]^1 && && && *{\bullet}  \ar@{.}[rr] \ar@{-}[ddll]_1&& *{\bullet} \ar@{-}[ddrr]^1 &&&& && *{\bullet} \ar@{-}[ddll]_1  \ar@{-}[rr]_2 &&  *{\bullet}\\
&&&&&&&&&&&&&&&&&&&&&\\
&& *{\bullet}  \ar@{-}[rr]_0 &&  *{\bullet}&& &&  &&  *{\bullet}\ar@{-}[rr]_0 &&  *{\bullet}&& &&&&  *{\bullet}\ar@{-}[rr]_0&&  *{\bullet}&& && && \\
}$$

\noindent Moreover the stabilizer of a point is, up to a conjugacy, $\langle \rho_0\rho_2, \rho_1\rho_2\rho_1\rangle$.
\end{proposition}
\begin{proof}
Let $G=\langle \rho_0,\rho_1,\rho_2\rangle$ be the group with the given permutation representation
graph. First we have $(\rho_0\rho_1)^4=(\rho_1\rho_2)^4=(\rho_0\rho_1)^2=(\rho_0\rho_1\rho_2\rho_1)^s=1$.
Hence $G$ is a subgroup of the group of symmetries of $\{4,4\}_{(s,0)}$, particularly $|G|\leq 8s^2$.
Consider the vertex $x$ of the graph. 
The stabilizer $G_x$ of $x$ contains $\rho_0\rho_2$, $\rho_1\rho_2\rho_1$ and $u^2$. Thus $|G_x|\geq 8s^2$.  Consequently the graph is a CPR graph of $\{4, 4\}_{(s,0)}$.\end{proof}

%----------------------------------------------------------------
%-----------------------------------------------------------------

\begin{lemma}\label{octo} 
A tilling on the torus by octagons and squares, as in the following example, is CPR graph of $\{4,4\}_{(s,0)}$ of degree $n=8ab$ with $s=$lcm$(a,b)$ ($s\geq 2$). 
Moreover the stabilizer of a point is, up to a conjugacy, $\langle u^a,v^b\rangle$.

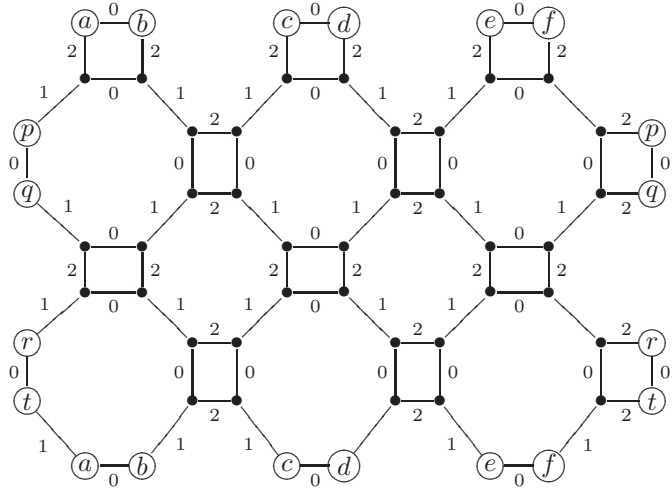
$$\xymatrix@-1.8pc{&&*+[o][F]{a}  \ar@{-}[rr]^0 \ar@{-}[dd]_2&& *+[o][F]{b} \ar@{-}[dd]^2&& &&  && *+[o][F]{c} \ar@{-}[rr]^0\ar@{-}[dd]_2 && *+[o][F]{d}\ar@{-}[dd]^2&& &&&& *+[o][F]{e}  \ar@{-}[dd]_2\ar@{-}[rr]^0&&*+[o][F]{f}\ar@{-}[dd]^2&& && && \\
&& && && && && && && && && && && && &&\\
&&*{\bullet} \ar@{-}[ddll]_1 \ar@{-}[rr]_0&& *{\bullet} \ar@{-}[ddrr]^1&& &&  && *{\bullet}  \ar@{-}[rr]_0\ar@{-}[ddll]_1 &&*{\bullet}\ar@{-}[ddrr]^1 && && && *{\bullet}  \ar@{-}[ddll]_1 \ar@{-}[rr]_0&&*{\bullet}\ar@{-}[ddrr]1 &&  
&& && \\
&&&&&&&&&&&&&&&&&&&&&\\
*+[o][F]{p}\ar@{-}[dd]_0 && && &&*{\bullet}  \ar@{-}[rr]^2\ar@{-}[dd]_0&&  *{\bullet}\ar@{-}[dd]^0 &&  && && *{\bullet}  \ar@{-}[rr]^2\ar@{-}[dd]_0 && *{\bullet}\ar@{-}[dd]^0  && && && *{\bullet}   \ar@{-}[rr]^2\ar@{-}[dd]_0&&*+[o][F]{p}  \ar@{-}[dd]^0 \\
&&&&&&&&&&&&&&&&&&&&&&\\ 
*+[o][F]{q}  \ar@{-}[ddrr]^1&& && &&*{\bullet}  \ar@{-}[rr]_2\ar@{-}[ddll]_1&&  *{\bullet} \ar@{-}[ddrr]^1 && && && *{\bullet}  \ar@{-}[rr]_2 \ar@{-}[ddll]_1&& *{\bullet} \ar@{-}[ddrr]^1 &&&& && *{\bullet} \ar@{-}[ddll]_1  \ar@{-}[rr]_2 &&  *+[o][F]{q}\\
&&&&&&&&&&&&&&&&&&&&&\\
&& *{\bullet}  \ar@{-}[rr]^0 \ar@{-}[dd]_2&&  *{\bullet}\ar@{-}[dd]^2&& &&  &&  *{\bullet}\ar@{-}[rr]^0\ar@{-}[dd]_2 &&  *{\bullet}\ar@{-}[dd]^2&& &&&&  *{\bullet} \ar@{-}[dd]_2\ar@{-}[rr]^0&&  *{\bullet}\ar@{-}[dd]^2&& && && \\
&& && && && && && && && && && && && &&\\
&&*{\bullet} \ar@{-}[ddll]_1 \ar@{-}[rr]_0&& *{\bullet} \ar@{-}[ddrr]^1&& &&  && *{\bullet}  \ar@{-}[rr]_0\ar@{-}[ddll]_1 &&*{\bullet}\ar@{-}[ddrr]^1 && && && *{\bullet}  \ar@{-}[ddll]_1 \ar@{-}[rr]_0&&*{\bullet}\ar@{-}[ddrr]1 &&  
&& && \\
&&&&&&&&&&&&&&&&&&&&&\\
*+[o][F]{r}\ar@{-}[dd]_0 && && &&*{\bullet}  \ar@{-}[rr]^2\ar@{-}[dd]_0&&  *{\bullet}\ar@{-}[dd]^0 &&  && && *{\bullet}  \ar@{-}[rr]^2\ar@{-}[dd]_0 && *{\bullet}\ar@{-}[dd]^0  && && && *{\bullet}   \ar@{-}[rr]^2\ar@{-}[dd]_0&&*+[o][F]{r}  \ar@{-}[dd]^0 \\
&&&&&&&&&&&&&&&&&&&&&&\\ 
*+[o][F]{t} \ar@{-}[ddrr]_1&& && &&*{\bullet}  \ar@{-}[rr]_2\ar@{-}[ddll]^1&&  *{\bullet} \ar@{-}[ddrr]_1&& && && *{\bullet}  \ar@{-}[rr]_2 \ar@{-}[ddll]1&& *{\bullet}\ar@{-}[ddrr]_1 &&&& && *{\bullet}   \ar@{-}[rr]_2 \ar@{-}[ddll]^1&&  *+[o][F]{t}\\
&&&&&&&&&&&&&&&&&&&&&&\\ 
&&*+[o][F]{a}  \ar@{-}[rr]_0 && *+[o][F]{b} && &&  && *+[o][F]{c} \ar@{-}[rr]_0&& *+[o][F]{d}&& &&&& *+[o][F]{e} \ar@{-}[rr]_0&&*+[o][F]{f}&& && && \\
}$$
In the example  above $a=3$ and $b=2$.
\end{lemma}
\begin{proof}
Let $G$ be the group generated by the involutions $\rho_0$, $\rho_1$ and  $\rho_2$ with the permutation given by a tilling of octagons and squares (as above) with $n=8ab$.
First we have $(\rho_0\rho_1)^4=(\rho_1\rho_2)^4=(\rho_0\rho_1)^2=(\rho_0\rho_1\rho_2\rho_1)^s=1$ with $s=lcm(a,b)$.
Consider any vertex $x$ of the tilling. We have that either $\langle u^a, v^b\rangle \leq G_x$ or  $\langle v^a, u^b\rangle \leq G_x$. Hence $|G_x|\geq \frac{s^2}{ab}$. Thus $|G|=8s^2$, which shows that the graph gives a permutation representation of a toroidal map $\{4,4\}_{(s,0)}$.
\end{proof}

%----------------------------------------------------------------------------------------------------------------------------------------------------------------------------------------------------------------------------------------------------------
%----------------------------------------------------------------------------------------------------------------------------------------------------------------------------------------------------------------------------------------------------------
\subsection{CPR graphs of small degree for $\{4,4\}_{(s,s)}$}

\begin{lemma}\label{4sss}
The following graph is a CPR graph of $\{4,4\}_{(s,s)}$ of degree $n=4s$ ($s\geq 2$).
$$
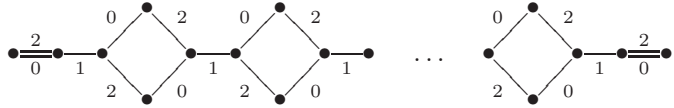
\xymatrix@-1.8pc{&& &&  && *{\bullet} \ar@{-}[ddrr] ^2\ar@{-}[ddll]_0 &&   && &&  *{\bullet} \ar@{-}[ddrr]^2\ar@{-}[ddll]_0 && && && &&     &&   *{\bullet}\ar@{-}[ddrr] ^2\ar@{-}[ddll]_0 && && &&\\
&& && && && && && &&  && && && && &&  &&  &&\\
*{\bullet} \ar@{=}[rr] _0^2&&*{\bullet} \ar@{-}[rr] _1&&*{\bullet} &&  &&*{\bullet} \ar@{-}[rr] _1&&*{\bullet} &&   &&*{\bullet} \ar@{-}[rr] _1&&*{\bullet}&&\ldots     &&*{\bullet} &&   &&*{\bullet} \ar@{-}[rr] _1&&*{\bullet} \ar@{=}[rr] _0^2&&*{\bullet} \\
&& && && && && && && && && && && && && && \\
&& &&  && *{\bullet} \ar@{-}[uurr] _0\ar@{-}[uull]^2 && &&   &&  *{\bullet} \ar@{-}[uurr] _0\ar@{-}[uull]^2 && && && &&  &&   *{\bullet} \ar@{-}[uurr] _0\ar@{-}[uull]^2 && && &&
}
$$

\noindent Moreover the stabilizer of a point is, up to a conjugacy, $\langle \rho_0\rho_2, \rho_0\rho_1\rho_2\rangle$.
\end{lemma}
\begin{proof}
Let $G=\langle \rho_0,\rho_1,\rho_2\rangle$ be the group with the given permutation representation graph ($s\geq 2$). First we have $(\rho_0\rho_1)^4=(\rho_1\rho_2)^4=(\rho_0\rho_1)^2=(\rho_0\rho_1\rho_2)^{2s}=1$.
Let $x$ be the second vertex starting from the left of the graph. The stabilizer $G_x$ of $x$ contains $\rho_0\rho_2$ and  $\rho_0\rho_1\rho_2$. Hence $G_x$ contains $\langle \rho_0\rho_2, \rho_0\rho_1\rho_2\rangle \cong D_{4s}$, thus  $|G_x|\geq 4s$ and therefore $|G|\geq 16s^2$.
\end{proof}

%-----------------------------------------------------------------------------------------------------------------------------

\begin{lemma}\label{8sss}
The following graph is a CPR graph of  $\{4,4\}_{(s,s)}$ of degree $8s$ ($s\geq 2$).

$$
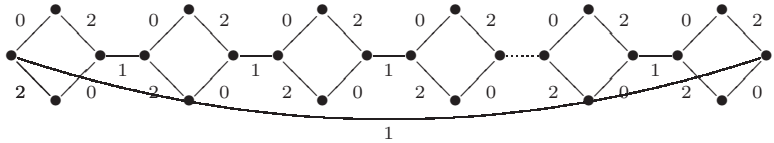
\xymatrix@-1.8pc{ && *{\bullet} \ar@{-}[ddrr] ^2\ar@{-}[ddll]_0 &&   && &&  *{\bullet} \ar@{-}[ddrr]^2\ar@{-}[ddll]_0 && && &&   *{\bullet}\ar@{-}[ddrr] ^2\ar@{-}[ddll]_0&&   && &&  *{\bullet} \ar@{-}[ddrr]^2\ar@{-}[ddll]_0 && && &&   *{\bullet}\ar@{-}[ddrr] ^2\ar@{-}[ddll]_0 && && && *{\bullet}\ar@{-}[ddrr] ^2\ar@{-}[ddll]_0 && \\
&& && && && &&  && && && && &&  && && && && && &&  &&  \\
*{\bullet}\ar@{-}@/_2pc/[rrrrrrrrrrrrrrrrrrrrrrrrrrrrrrrrrr]_1&& &&*{\bullet} \ar@{-}[rr] _1 &&*{\bullet} && &&*{\bullet}  \ar@{-}[rr] _1&& *{\bullet} && &&*{\bullet} \ar@{-}[rr] _1 &&*{\bullet} && &&*{\bullet}  \ar@{.}[rr] && *{\bullet}&&   &&*{\bullet} \ar@{-}[rr] _1&&*{\bullet} && &&*{\bullet} \\
 && && && && && && && && && && && && \\
  && *{\bullet} \ar@{-}[uurr] _0\ar@{-}[uull]^2  \ar@{-}[uull]^2  && &&   &&  *{\bullet} \ar@{-}[uurr] _0\ar@{-}[uull]^2 && && &&    *{\bullet} \ar@{-}[uurr] _0\ar@{-}[uull]^2  && &&   &&  *{\bullet} \ar@{-}[uurr] _0\ar@{-}[uull]^2 && && &&    *{\bullet} \ar@{-}[uurr] _0\ar@{-}[uull]^2 && && && *{\bullet} \ar@{-}[uurr] _0\ar@{-}[uull]^2 &&
}
$$
\noindent Moreover the stabilizer of a point is, up to a conjugacy, $\langle g\rangle\rtimes \langle \rho_1\rangle $.
\end{lemma}
\begin{proof}
Let $G=\langle \rho_0,\rho_1,\rho_2\rangle$ be the group with the given permutation representation graph ($s\geq 2$). First we have $(\rho_0\rho_1)^4=(\rho_1\rho_2)^4=(\rho_0\rho_1)^2=(\rho_0\rho_1\rho_2)^{2s}=1$.
Consider a vertex $x$ that is not fixed by $\rho_1$. The permutation $g=(\rho_0\rho_1\rho_2)^2$ is in the stabilizer of $x$. Hence $G_x$ thus $|G_x|\geq s$ and therefore $|G|\geq 16s^2$, which proves that indeed $G$ is the group of the toroidal regular map $\{4,4\}_{(s,s)}$. 
\end{proof}

%--------------------------------------------------------------------------------------------------------------------------------------------------------------------------------------------------------------------------------------
%--------------------------------------------------------------------------------------------------------------------------------------------------------------------------------------------------------------------------------------
%--------------------------------------------------------------------------------------------------------------------------------------------------------------------------------------------------------------------------------------
%--------------------------------------------------------------------------------------------------------------------------------------------------------------------------------------------------------------------------------------
%--------------------------------------------------------------------------------------------------------------------------------------------------------------------------------------------------------------------------------------
%--------------------------------------------------------------------------------------------------------------------------------------------------------------------------------------------------------------------------------------

\section{The maps of type $\{3,6\}$}

 In this section we consider the regular maps of type $\{3,6\}$, we determine all possible degrees for these maps and we give CPR graphs for some of those degrees.

\subsection{The possible degrees for the map $\{3,6\}_{(s,0)}$}
 The smallest map of type $\{3,6\}$ has only one vertex, three edges and two faces. It does have a faithful representation of degree 6 (on the darts) and of degree 12 (on the flags) but not on the vertices, edges and faces. The map  $\{3,6\}_{(2,0)}$ has a faithful permutation representation on the edges, faces, darts and flags but not on the set of vertices. Any other map  $\{3,6\}_{(s,0)}$ with $s>2$ has faithful permutation representation of degrees $s^2$, $2s^2$, $3s^2$, $4s^2$, $6s^2$ and $12s^2$, as the dihedral groups $\langle \rho_i, \rho_j\rangle$ and its subgroups are core-free, with $i,j\in\{0,1,2\}$.
In what follows we describe core-free subgroups of the groups of symmetries $G$ of the map $\{3,6\}_{(s,0)}$ and we calculate in each case the index of that subgroup, corresponding to the degree of a faithful permutation representation of $G$.

%--------------------------------------------------------------------------------------------------------------------------------------------------------------------------------------------------------------------------------------

\begin{proposition}\label{3core1}
 Let $G$ be the group of $\{3,6\}_ {(s,0)}$ ($s\geq 2$).
\begin{enumerate}
\item $H=\langle u^a, v^b\rangle$ is core-free and $|G:H|=12ab$, where $s=lcm(a,b)$.
\item If $s=lcm(a,b)$ and $s\neq ab$ then $H=\langle u^a, v^b\rangle\rtimes \langle \rho_0\rho_2 \rangle$ is core-free and $|G:H|=6ab$.
\item If $d$ is a divisor of $s$ then $H=\langle u^d \rangle\rtimes \langle \rho_0,\rho_2 \rangle$  and $H'=\langle u^d \rangle\rtimes \langle \rho_0\rho_2 \rangle$ 
 are core-free. Moreover $|G:H|= 3ds$ and  $|G:H'|= 6ds$.
\end{enumerate}
\end{proposition}

\begin{proof}

 The proof of (1) is similar the proof of Proposition~\ref{core} (1). Let us now prove (2) and (3) separately.

(2) Suppose that $x\in H\cap H^{\rho_1}=\langle u^a, v^b\rangle\rtimes \langle \rho_0\rho_2\rangle\cap \langle u^b, v^a\rangle\rtimes \langle (\rho_0\rho_2)^{\rho_1}\rangle$.
If $x\notin T$ then $(\rho_0\rho_2)^i((\rho_0\rho_2)^{\rho_1})^j\in K:=\langle u^a, v^b, u^b, v^a\rangle$ for some $i,\,j\in\{0,1\}$. 
This is only possible for $i=j=1$, implying that $u\in K$, which is only possible if $s=ab$, which is not the case.
Thus  $x\in T$ and therefore, $x=1$.
The order of  $H$ is $\frac{12s^2}{ab}$ thus $|G:H|=6ab$.

(3)  Suppose that $x\in H\cap H^{\rho_1} = \langle u^d\rangle \rtimes \langle \rho_0 , \rho_2\rangle \cap  \langle v^d\rangle \rtimes \langle \rho_0^{\rho_1} , \rho_2^{\rho_1}\rangle$. 
 If $x\notin T$ then there is a nontrivial element $\rho_0^i\rho_2^j(\rho_0^{\rho_1})^k(\rho_2^{\rho_1})^l\in K:= \langle u^d,v^d\rangle$ for some $i,j,k,l\in \{0,1\}$. Which implies that $(i,j,k,l)=(1,1,1,1)$ and that $j\in K$, which is only possible if  $d=1$.  If $x\in T$, then $x\in \langle u^d\rangle \cap  \langle v^d\rangle=\{1\}$. Hence $H$ is core free when $d\neq 1$.

Consider  $d=1$. In this case $u^{-1}\rho_0\rho_2=v^{-1}\rho_0^{\rho_1}\rho_2^{\rho_1}=(\rho_1\rho_2)^3$ is the unique nontrivial element in $H\cap H^{\rho_1}$, hence $H\cap H^{\rho_1}=\langle (\rho_1\rho_2)^3 \rangle$ which is not a normal subgroup of $G$. Thus also in this case $H$ is core-free.  As $H'$ is a subgroup of $H$, it is also core-free.

The order of $H$ is $\frac{4s}{d}$  and the order of $H'$ is  $\frac{2s}{d}$ thus $|G:H|=3ds$ and  $|G:H|=6ds$. 
\end{proof}

%--------------------------------------------------------------------------------------------------------------------------------------------------------------------------------------------------------------------------------------

\begin{corollary}\label{3cores}
 Let $s>2$. There exists faithful permutation representations of the group of the toroidal map $\{3,6\}_{(s,0)}$ for $n\in\{s^2,\,2s^2, \,3ds,\,4s^2,\,6ab,\,12ab\}$. \end{corollary}
\begin{proof}
This is a consequence of Proposition~\ref{3core1}.
%--------------------------------------------------------------------------------------------------------------------------------------------------------------------------------------------------------------------------------------
\end{proof}

Note that to extend Corollary~\ref{3cores} to the case $s=2$ only the possibility $n=s^2\ (n=4)$ needed to be excluded from the set of possibilities for the degree $n$.

In what follows we prove that the degrees given in Corollary~\ref{3cores} are the only  possible degrees for the group of symmetries of the map $\{3,6\}_{(s,0)}$ with  $s\geq 2$.
By Lemma~\ref{Gimp} we now consider that $G$ is embedded into $S_k\wr S_m$   where $n=km$ with $m\in\{2,\,3,\,4\}$ being the number of orbits of $T=\langle u,v\rangle$ where $u$ and $v$ are as in Section~\ref{back}. In what follows we also consider the translation $t$ defined in Section~\ref{back}.

\begin{proposition}\label{m=2}
 If $m=2$ then $k=s^2$.
\end{proposition}

\begin{proof}
 Let $m=2$. Let $B_1$ and $B_2$ denote the two blocks. The following graphs represent all the possible block actions (determined by the fact that $u$ and $v$ fix the blocks setwisely).
 
 $$
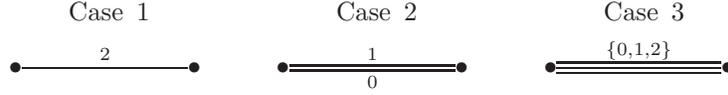
\xymatrix@-1.8pc{ 
 && \mbox{ Case }\, 1 &&&&&&&& \mbox{ Case } \,2 &&&&&&&& \mbox{ Case }\,3\\ \\
 *{\bullet} \ar@{-}[rrrr]^2 && && *{\bullet} &&&& *{\bullet} \ar@2{-}[rrrr]^1_0  &&&& *{\bullet} && && *{\bullet} \ar@3{-}[rrrr]^{\{0,1,2\}} && && *{\bullet}\\
 }$$
Let $u_1$ and $u_2$ be the action of $u$ on the blocks $B_1$ and $B_2$ respectively.
In any of the three cases the action of $u_1$ and $u_2$ have the same cyclic decomposition. 
Hence, by Proposition~\ref{cxc}, $k\in\{s,s^2\}$. Now we assume that $k=s$, then $u=u_1u_2$, $v=v_1v_2$ and $t=t_1t_2$ with $u_i$, $v_i$ and $t_i$ being cycles of length $s$.  
We now analise each case separately. 
 
 Case 1: 
 We have that  $t_1=u_1^{\alpha_1}$ and $t_2=u_2^{\alpha_2}$ for some integers $\alpha_1$ and $\alpha_2$ coprime with $s$.
 Thus, $1= tt^{-1} = t(t^{\rho_1}) = u_1^{\alpha_1}u_2^{\alpha_2} 
 (u_1^{\rho_1})^{\alpha_1}(u_2^{\rho_1})^{\alpha_2} =
 u_1^{\alpha_1}u_2^{\alpha_2}v_1^{\alpha_1}v_2^{\alpha_2} =
 (u_1v_1)^{\alpha_1}(u_2v_2)^{\alpha_2}$.
 But this implies that $u=v^{-1}$, a contradiction.
  
 Case 2:
 We have that  $(uv)^{\rho_1\rho_0\rho_2}=(uv)^{-1}$. Let $uv=(uv)_1(uv)_2$ with $(uv)_i$  being the respective actions of $uv$ in blocks $B_i\ (i\in\{1,2\})$. There  exist $\alpha_1$ and $\alpha_2$, coprime with $s$, such that  $(uv)_1=u_1^{\alpha_1}$  and $(uv)_2=u_2^{\alpha_2}$.  Thus $1 = (uv)(uv)^{-1} = (uv)(uv)^{\rho_1\rho_0\rho_2} = u_1^{\alpha_1}u_2^{\alpha_2} 
 (u_1^{\rho_1\rho_0\rho_2})^{\alpha_1}(u_2^{\rho_1\rho_0\rho_2})^{\alpha_2} =
 u_1^{\alpha_1}u_2^{\alpha_2}v_1^{-\alpha_1}v_2^{-\alpha_2} =
 (u_1v_1^{-1})^{\alpha_1}(u_2v_2^{-1})^{\alpha_2} = t_1^{-\alpha_1}t_2^{-\alpha_2}$.
 But this implies that $ \alpha_i = 0\, (mod\ s)$ for $i\in\{1,2\}$, a contradiction.

 Case 3:
 Let $v_1=u_1^{\alpha_1}$ and $v_2=u_2^{\alpha_2}$, with $\alpha_i$ with $i\in\{1,2\}$ being co-prime with $s$. Note that $u_1^{\rho_0}=u_2^{-1}$, $u_1^{\rho_1}=v_2$ and $u_2^{\rho_1}=v_1$.
 We have that $1 = tt^{-1} = vu^{-1}(v^{-1})^{\rho_0} = u_1^{\alpha_1}u_2^{\alpha_2} u_1^{-1}u_2^{-1} (u_1^{\rho_0})^{-\alpha_1}(u_2^{\rho_0})^{-\alpha_2} =
 u_1^{\alpha_1+\alpha_2-1}u_2^{\alpha_1+\alpha_2-1} = u^{\alpha_1+\alpha_2-1}$. Hence $\alpha_1 + \alpha_2 -1 = 0\, (mod\ s)$.
 In addition $1= uu^{-1} = v^{\rho_1}u^{-1} = (u_1^{\rho_1})^{\alpha_1}(u_2^{\rho_1})^{\alpha_2} u_1^{-1}u_2^{-1} = v_2^{\alpha_1} v_1^{\alpha_2} u_1^{-1}u_2^{-1} =
 u_1^{\alpha_1\alpha_2-1} u_2^{\alpha_1\alpha_2-1}$. Thus  $\alpha_1\alpha_2 -1= 0\, (mod\ s)$.
 With this, we conclude that $\alpha_1^2-\alpha_1+1 = 0\,(mod\ s)$ and $\alpha_2^2-\alpha_2+1 = 0\, (mod\ s)$, implying that $s$ is odd. 
 As $(\rho_1\rho_2)^3$ has a fixed point (otherwise $s$ would be even) at least one of the components of the CPR-graph $\mathcal{G}_{1,2}$ must be of the one of two following forms.
 \begin{center}
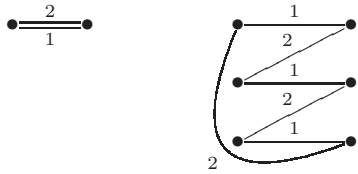
  \begin{tabular}{lccr}
  
$\xymatrix@-1.9pc{ 
 *{\bullet} \ar@2{-}[rrrr]^2_1 && && *{\bullet}\\
 \\
 \\
 \\ \\}$

&$\qquad$&
 $\xymatrix@-1.9pc{
 *{\bullet} \ar@{-} @/_3pc/ [rrrrrrdddddd]_2 \ar@{-}[rrrrrr]^1 && && && *{\bullet} \\ \\ \\
 *{\bullet} \ar@{-}[rrrrrr]^1 \ar@{-}[rrrrrruuu]^2 && && && *{\bullet}\\ \\ \\
 *{\bullet} \ar@{-}[rrrrrr]^1 \ar@{-}[rrrrrruuu]^2 && && && *{\bullet}
 }$
 
  \end{tabular}
\end{center}
 Let $\delta = (\rho_1\rho_2)^3$. As $u^\delta = u^{-1}$ and $s$ is odd the action of $u$ and $\delta$ on a block, say $B_1$, can be described by the following graph. 
 $$
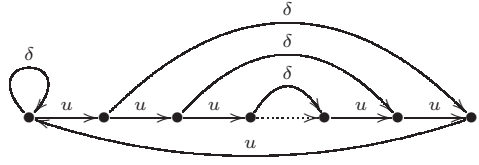
\xymatrix@-1.9pc{
 *{\bullet} \ar@{->}[rrrr]^u \ar@(ul,ur)^\delta \ar@{<-}@/_1.2pc/[rrrrrrrrrrrrrrrrrrrrrrrr]^u &&&& *{\bullet} \ar@{->}[rrrr]^u \ar@{->}@/^3pc/ [rrrrrrrrrrrrrrrrrrrr]^\delta &&&& *{\bullet} \ar@{->}[rrrr]^u \ar@{->}@/^2pc/ [rrrrrrrrrrrr]^\delta&&&& *{\bullet} \ar@{.>}[rrrr] \ar@{->}@/^1pc/ [rrrr]^\delta&&&& *{\bullet} \ar@{->}[rrrr]^u &&&& *{\bullet} \ar@{->}[rrrr]^u &&&& *{\bullet} \\ \\}$$ 
 Hence $\delta$ has exactly one fixed point on $B_1$. Note that if  $\mathcal{G}_{1,2}$  contains a cycle of size six (one component is the second possibility given before) then $\delta$ fixes three points on each block, a contradiction. We may then conclude that there is a point $x\in B_1$ that is fixed by $\rho_1\rho_2$.
 In addition, $\rho_0\rho_2$ must also have a fixed point, or otherwise $s$ would be even. Thus $(\rho_0\rho_2)^{u^i}$ fixes $x$ for some $i\in\{0,\ldots,s-1\}$.
 Moreover, $v^{\alpha_2}u^{-1} = u_2^{\alpha_2^2-1}$ fixes $B_1$ point-wisely.
 Hence, the group $H=\langle v^{\alpha_2}u^{-1}, (\rho_0\rho_2)^{u^i}, \rho_1\rho_2\rangle$ must be a subgroup of $G_x$. Let $K=\langle v^{\alpha_2}u^{-1}\rangle$. We have that $|K|=s$ and, as $u\notin K$,
 $$|\{K,\,K(\rho_0\rho_2)^{u^i},\, K \rho_1\rho_2,\, K (\rho_1\rho_2)^2,\,  K (\rho_1\rho_2)^3,\, K (\rho_1\rho_2)^{-1},\, K (\rho_1\rho_2)^{-2}\}|=7.$$
Indeed if we assume that $K(\rho_0\rho_2)^{u^i}=K(\rho_1\rho_2)^l$ for some $l\in\{0,\ldots, 5\}$, which gives $(\rho_1\rho_2)^l\rho_0\rho_2\in \langle u^i, v^{\alpha_2}u^{-1}\rangle$, a contradiction.
Hence $|H|\geq 7s$, which is also not possible as $n>2s$ and $|G|=12s^2$.

 \end{proof}
  %--------------------------------------------------------------------------------------------------------------------------------------------------------------------------------------------------------------------------------------
%--------------------------------------------------------------------------------------------------------------------------------------------------------------------------------------------------------------------------------------
 
 \begin{proposition}\label{m=3}
 If  $m=3$ then $k=as$ for some divisor $a$ of $s$.
\end{proposition}

\begin{proof}
In this case it can be easily checked in GAP \cite{GAP4} that there is only one possibility for the action of $G$ on the blocks (with $T$ fixing the blocks) described by the following graph.  
 $$\xymatrix@-1.8pc{ 
 *{\bullet} \ar@{-}[rrrr]^1 && && *{\bullet} \ar@2{-}[rrrr]^2_0  &&&& *{\bullet} \\
 }$$
 We may assume that $k\neq s$, and therefore, by Proposition~\ref{cxc}, both  $u$ and $v$ act intransitively on the blocks and $k=ab$ with $s=lcm(a,b)$.
 Let $B_1$ denote the block corresponding to the vertex on the left, $B_2:=B_1\rho_1$ and $B_3:=B_2\rho_0 = B_2\rho_2 $ and let $u_i$, $v_i$ and $t_i$ denote the actions of $u$, $v$ and $t$ on $B_i$.
Suppose that $|u_1|=a$ and $|v_1|=b$. As $u^{\rho_1}=v$, $|u_2|=b$ and $|v_2|=a$. As $u^{\rho_0}=u^{-1}$ and $u^{\rho_2}=u$, $|u_3|=b$. As $u$ and $v$ have the same cyclic decomposition, $|v_3|=b$.
On the other hand $v^{\rho_0}=t$, hence $|v_3|=|t_2|=lcm(a,b)$. Therefore $lcm(a,b)=b=s$ which implies that $k=as$ for some divisor $a$ of $s$.  
\end{proof}

%--------------------------------------------------------------------------------------------------------------------------------------------------------------------------------------------------------------------------------------
%--------------------------------------------------------------------------------------------------------------------------------------------------------------------------------------------------------------------------------------

\begin{proposition}\label{m=4}
 If $m=4$ then $k=s^2$.
\end{proposition}
\begin{proof}
It can be checked with GAP \cite{GAP4} that in this case there is only one possibility for the action of $G$ on the blocks given by the following diagram. 
  $$
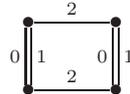
\xymatrix@-1.8pc{ 
 *{\bullet} \ar@2{-}[ddd]^1_0 \ar@{-}[rrrr]^2 && && *{\bullet}\ar@2{-}[ddd]^1_0  \\ \\ \\
 *{\bullet} \ar@{-}[rrrr]^2 && && *{\bullet} 
 }$$
Suppose that $k\neq s$. Then $u$ and $v$ are intransitive on the blocks and $k=ab$ with $s=lmc(a,b)$. The existence of a double with label $0$ and $1$ implies that $a=b$.
Thus we have only to rule out the case where $u$ and $v$ are both transitive on the blocks of size $s$.

Suppose then that $k = s$. In this case neither $u$ nor $v$ have fixed points. Let $u_i$ and $v_i$ be the actions of $u$ and $v$ on the block $B_i$, $i=1,\ldots,4$. Let $B_2=B_1\rho_0=B_1\rho_1$, $B_3=B_1\rho_2$ and $B_4=B_3\rho_0=B_3\rho_1$.
Let $v=u_1^{\alpha_1} u_2^{\alpha_2} u_3^{\alpha_3} u_4^{\alpha_4} $ where $\alpha_i$,  $i=1,\ldots,4$  is coprime with $s$. From the relation $v^{\rho_0}=t$ we get that $\alpha_1-1=\alpha_2\, (mod\, s)$ and $\alpha_3-1=\alpha_4\, (mod\, s)$. From $v^{\rho_1}=u$ we get $\alpha_1\alpha_2=1\, (mod\, s)$ and $\alpha_3\alpha_4=1\, (mod\, s)$. Particularly $\alpha_1(\alpha_1-1)=1\, (mod\, s)$ and $s$ is even. 
As  $u_1^{\rho_0\rho_1}=v_1^{-1}=u_1^{-\alpha_1}$, we have $u_1=u_1^{(\rho_0\rho_1)^3}=u_1^{-\alpha_1^3}$, which implies that $\alpha_1^3=-1\, (mod\, s)$. As in addition $\alpha_1(\alpha_1-1)=1\, (mod\, s)$, we get $2\alpha_1=-2\, (mod\, s)$, which gives $\alpha_1=-1\, (mod\, s)$, a contradiction .
\end{proof}

\begin{theorem}\label{T3}
Let $s>2$. There exists a CPR graph of a toroidal map $\{3,6\}_{(s,0)}$ with $n$ vertices if and only if $n\in\{s^2, \,2s^2, \,3ds,\,4s^2,\,6ab,\,12ab\}$ where $s=lcm(a,b)$ and $d$ is a divisor of $s$. \end{theorem}
\begin{proof}
This is a consequence of  Lemmas~\ref{Ttran} and ~\ref{Gimp}, Corollary \ref{3cores}, Propositions~\ref{m=2}, \ref{m=3} and \ref{m=4}.
\end{proof}

%-----------------------------------------------------------------------------------------------------------------------------------------------------------------------------------------------------------------------------------------------------------------------------------------------------------
%-----------------------------------------------------------------------------------------------------------------------------------------------------------------------------------------------------------------------------------------------------------------------------------------------------------

\subsection{The possible degrees for the map  $\{3,6\}_{(s,s)}$}

In this section we determine the degrees of $\{3,6\}_{(s,s)}$ using the degrees of $\{3,6\}_{(s,0)}$ and  $\{3,6\}_{(2s,0)}$, given in Theorem~\ref{T3}.
Let $G$ be the group of $\{3,6\}_{(s,s)}$.
Consider first the particular case $s=1$. In this case $(\rho_1\rho_2)^2$ is a normal subgroup of $G$ hence the action on the set of vertices is not faithful.
On the other hand the group  $\langle \rho_0,\rho_1\rangle$ is core-free, hence $G$ has a faithful action on the faces. All the other possible degrees for  $\{3,6\}_{(1,1)}$ are the degrees of  $\{3,6\}_{(1,0)}$ multiplied by three.
Contrarily to what happen with the map $\{3,6\}_{(2,0)}$, the map $\{3,6\}_{(2,2)}$ has a faithful action on the set of vertices. The remaining degrees for $\{3,6\}_{(2,2)}$  are obtained multiplying by three each degree of the map $\{3,6\}_{(2,0)}$. In what follows we deal with the case $s>2$.

\begin{theorem}\label{T4}
Let $s\geq 2$. There exists a CPR graph of a toroidal map $\{3,6\}_{(s,s)}$ with $n$ vertices if and only if $n\in\{3s^2, \,6s^2, \,9ds,\,12s^2,\,18ab,\,36ab\}$ with $s=lcm(a,b)$ and $d$ a divisor of $s$. \end{theorem}
\begin{proof}
By Theorem~\ref{T3} and Lemma~\ref{degreesss}, using the embedding of $\{3,6\}_{(s,0)}$ into  $\{3,6\}_{(s,s)}$ there are faithful permutation representations with the degrees given in this theorem. 
By Theorem~\ref{T3} we have that the possible degrees for $\{3,6\}_{(3s,0)}$  are  $9s^2, \,18s^2, \,9\delta s,\,36s^2,\,6\alpha\beta,\,12\alpha\beta$ with $\delta$ dividing $3s$ and $lcm(\alpha,\beta)=3s$. Moreover each of these degrees is attained when the translation group $H=\langle u,v\rangle$ of order $(3s)^2$ has $m=1,\, 2,\, 3,\, 4,\, 6$ or $12$ orbits, respectively. Now using the embedding of $\{3,6\}_{(s,s)}$ into  $\{3,6\}_{(3s,0)}$ we divide each of these degrees by three and we get that the possible degrees for $\{3,6\}_{(s,s)}$ are in the set $\{3s^2, \,6s^2, \,3\delta s,\,12s^2,\,2\alpha\beta,\,4\alpha\beta\}.$

Let  $H=\langle u,v\rangle$ be the group $G$ of translations of  $\{3,6\}_{(3s,0)}$. The map   $\{3,6\}_{(3s,0)}$ contain three copies of  the map $\{3,6\}_{(s,s)}$ whose group of translation is $T=\langle g,\, h\rangle$ where $g\equiv uv$ and $g\equiv u^{-1}v$ where $x\equiv y$ means  $xy^{-1}\in \langle (uv)^s\rangle$.
Note that when $H$ has $m$ orbits on $3n$ points, the number of orbits of $T$ on $n$ points must be greater than $m$. Recall also that $m$ must be a divisor of $36$ by Lemma~\ref{Gimp}.

Let us now consider separately the cases: (1) $n=3\delta s$ ($m=3$); (2) $n=2\alpha\beta$ ($m=6$);  and (3) $n=4\alpha\beta$  ($m=12$) . 

(1) Let $n=3\delta s$ with $\delta$ dividing $3s$. In this case $H$ has $3$ orbits. But then $T$ must have at least $4$ orbits, hence $n\geq 4ab$ with $lcm(a,b)=s$. Hence $\delta$ cannot be a divisor of $s$. Consequently $n=9ds$ with $d$ being a divisor of $s$.

(2) Now let $n=2\alpha\beta$.  In this case $H$ has $6$ orbits, hence $T$ has at least $9$ orbits. Thus $\alpha=3a$ and $\beta=3b$ for some $a$ and $b$ with $lcm(a,b)=s$, otherwise (if either $\alpha$ or $\beta$ divides $s$) $n$ would be too small.

(3) Finally let $n=4\alpha\beta$. In this case $H$ has $12$ orbits, hence $T$ has at least $18$ orbits. As before $\alpha=3a$ and $\beta=3b$ for some $a$ and $b$ with $lcm(a,b)=s$ and therefore $n=36ab$ with $lcm(a,b)=s$.
\end{proof}

%----------------------------------------------------------------------------------------------------------------------------------------------------------------------------------------------------------------------------------------------------------------------------------------------------------
%-----------------------------------------------------------------------------------------------------------------------------------------------------------------------------------------------------------------------------------------------------------------------------------------------------------

\subsection{CPR graphs of small degree for $\{3,6\}_{(s,0)}$}
 We start by giving a CPR graph of minimal degree for  $\{3,6\}_{(s,0)}$ when $s\geq 3$. 

\begin{proposition}\cite{FLW2}\label{36s03s}
Let $s\geq 3$. The following graphs are CPR graphs of $\{3,6\}_{(s,0)}$ of degree $3s$. Moreover the stabilizer of a point is, up to a conjugacy, $\langle u\rangle \rtimes \langle \rho_0,\rho_2\rangle$.

\begin{center}
$s$ even
$$\xymatrix@-1.8pc{
 &&&&&*{\bullet}&&&&&&&&&&*{\bullet}&&&&&&&&&&&&&&&*{\bullet}\\
 \\
 &*{\bullet}\ar@{-}[rr]_1&&*{\bullet}\ar@{-}[urur]^0\ar@{-}[drdr]_2&&&&*{\bullet}\ar@{-}[lulu]_2\ar@{-}[ldld]^0\ar@{-}[rr]_1&&*{\bullet}\ar@{-}[rr]_2 &&*{\bullet}\ar@{-}[rr]_1 &&*{\bullet}\ar@{-}[urur]^0\ar@{-}[drdr]_2&&&&*{\bullet}\ar@{-}[lulu]_2\ar@{-}[ldld]^0 \ar@{-}[rr]_1 &&*{\bullet}\ar@{-}[rr]_2&&*{\bullet}\ar@{-}[rr]_1&&*{\bullet}\ar@{.}[rrrrr]&&&&&*{\bullet}\ar@{-}[urur]^0\ar@{-}[drdr]_2&&&&*{\bullet}\ar@{-}[lulu]_2\ar@{-}[ldld]^0\ar@{-}[rr]_1&&*{\bullet}\\
 \\
 &&&&&*{\bullet}&&&&&&&&&&*{\bullet}&&&&&&&&&&&&&&&*{\bullet}\\
 }$$
 
$s$ odd
 $$\xymatrix@-1.8pc{
 &&&&&*{\bullet}&&&&&&&&&&*{\bullet}&&&&&&&&&&&&&&&*{\bullet}\\
 \\
 &*{\bullet}\ar@{-}[rr]_1&&*{\bullet}\ar@{-}[urur]^0\ar@{-}[drdr]_2&&&&*{\bullet}\ar@{-}[lulu]_2\ar@{-}[ldld]^0\ar@{-}[rr]_1&&*{\bullet}\ar@{-}[rr]_2 &&*{\bullet}\ar@{-}[rr]_1 &&*{\bullet}\ar@{-}[urur]^0\ar@{-}[drdr]_2&&&&*{\bullet}\ar@{-}[lulu]_2\ar@{-}[ldld]^0 \ar@{-}[rr]_1 &&*{\bullet}\ar@{-}[rr]_2&&*{\bullet}\ar@{-}[rr]_1&&*{\bullet}\ar@{.}[rrrrr]&&&&&*{\bullet}\ar@{-}[urur]^0\ar@{-}[drdr]_2&&&&*{\bullet}\ar@{-}[lulu]_2\ar@{-}[ldld]^0\ar@{-}[rr]_1&&*{\bullet}\ar@{-}[rr]_2&&*{\bullet}\ar@{-}[rr]_1&&*{\bullet}\ar@2{-}[rr]_0^2&&*{\bullet}\\
 \\
 &&&&&*{\bullet}&&&&&&&&&&*{\bullet}&&&&&&&&&&&&&&&*{\bullet}\\
 }$$
\end{center}
\end{proposition}
%-----------------------------------------------------------------------------------------------------------------------------------------------------------------------------------------------------------------------------------------------------------------------------------------------------------
The following propositions give permutation representations of degree $6s$. The proofs are similar to those of Section~\ref{CPR44} and for this reason will be omitted.
\begin{proposition}
Let $s$ be even. Then the following graphs are CPR graphs of $\{3,6\}_{(s,0)}$ of degree $6s$. Moreover the stabilizer of a point is, up to a conjugacy,  $\langle u^2\rangle \rtimes \langle \rho_0 , \rho_2\rangle$.

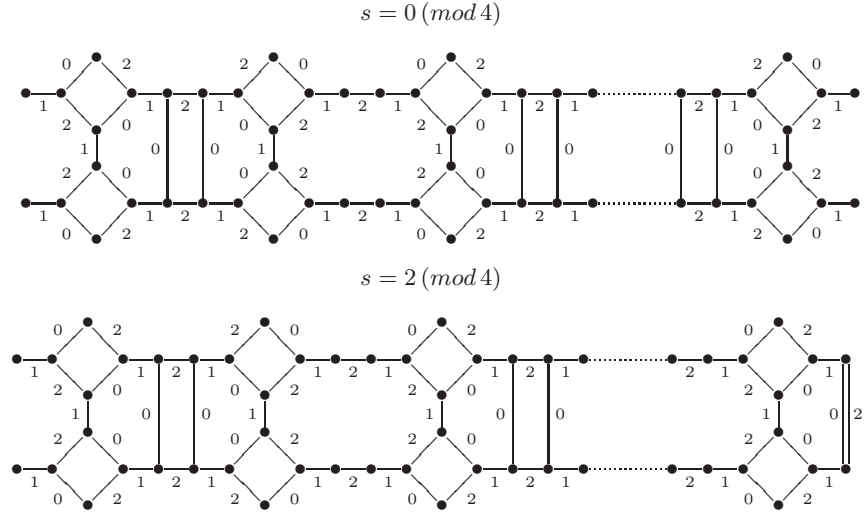
\begin{center}
\begin{small}
$s=0\, (mod \,4)$
$$\xymatrix@-1.9pc{
 &&&&&*{\bullet}&&&&&&&&&&*{\bullet}&&&&&&&&&&*{\bullet}&&&&&&&&&&&&&&&&&&&*{\bullet}\\
 \\
 &*{\bullet}\ar@{-}[rr]_1&&*{\bullet}\ar@{-}[urur]^0\ar@{-}[drdr]_2&&&&*{\bullet}\ar@{-}[lulu]_2\ar@{-}[ldld]^0\ar@{-}[rr]_1&&*{\bullet}\ar@{-}[rr]_2\ar@{-}[dddddd]_0 &&*{\bullet}\ar@{-}[dddddd]^0\ar@{-}[rr]_1 &&*{\bullet}\ar@{-}[urur]^2\ar@{-}[drdr]_0&&&&*{\bullet}\ar@{-}[lulu]_0\ar@{-}[ldld]^2 \ar@{-}[rr]_1 &&*{\bullet}\ar@{-}[rr]_2&&*{\bullet}\ar@{-}[rr]_1&&*{\bullet}\ar@{-}[urur]^0\ar@{-}[drdr]_2&&&&*{\bullet}\ar@{-}[lulu]_2\ar@{-}[ldld]^0\ar@{-}[rr]_1&&*{\bullet}\ar@{-}[rr]_2\ar@{-}[dddddd]_0 &&*{\bullet}\ar@{-}[dddddd]^0\ar@{-}[rr]_1 &&*{\bullet}\ar@{.}[rrrrr]&&&&&*{\bullet}\ar@{-}[rr]_2\ar@{-}[dddddd]_0 &&*{\bullet}\ar@{-}[dddddd]^0\ar@{-}[rr]_1 &&*{\bullet}\ar@{-}[urur]^2\ar@{-}[drdr]_0&&&&*{\bullet}\ar@{-}[lulu]_0\ar@{-}[ldld]^2\ar@{-}[rr]_1&&*{\bullet}\\
 \\
 &&&&&*{\bullet}\ar@{-}[dd]_1&&&&&&&&&&*{\bullet}\ar@{-}[dd]_1&&&&&&&&&&*{\bullet}\ar@{-}[dd]_1&&&&&&&&&&&&&&&&&&&*{\bullet}\ar@{-}[dd]_1\\
 \\
 &&&&&*{\bullet}&&&&&&&&&&*{\bullet}&&&&&&&&&&*{\bullet}&&&&&&&&&&&&&&&&&&&*{\bullet}\\
 \\
&*{\bullet}\ar@{-}[rr]_1&&*{\bullet}\ar@{-}[urur]^2\ar@{-}[drdr]_0&&&&*{\bullet}\ar@{-}[lulu]_0\ar@{-}[ldld]^2\ar@{-}[rr]_1&&*{\bullet}\ar@{-}[rr]_2 &&*{\bullet}\ar@{-}[rr]_1 &&*{\bullet}\ar@{-}[urur]^0\ar@{-}[drdr]_2&&&&*{\bullet}\ar@{-}[lulu]_2\ar@{-}[ldld]^0 \ar@{-}[rr]_1 &&*{\bullet}\ar@{-}[rr]_2&&*{\bullet}\ar@{-}[rr]_1&&*{\bullet}\ar@{-}[urur]^2\ar@{-}[drdr]_0&&&&*{\bullet}\ar@{-}[lulu]_0\ar@{-}[ldld]^2\ar@{-}[rr]_1&&*{\bullet}\ar@{-}[rr]_2 &&*{\bullet}\ar@{-}[rr]_1 &&*{\bullet}\ar@{.}[rrrrr]&&&&&*{\bullet}\ar@{-}[rr]_2 &&*{\bullet}\ar@{-}[rr]_1 &&*{\bullet}\ar@{-}[urur]^0\ar@{-}[drdr]_2&&&&*{\bullet}\ar@{-}[lulu]_2\ar@{-}[ldld]^0\ar@{-}[rr]_1&&*{\bullet}\\
 \\
 &&&&&*{\bullet}&&&&&&&&&&*{\bullet}&&&&&&&&&&*{\bullet}&&&&&&&&&&&&&&&&&&&*{\bullet}\\
 }$$

 $s=2\, (mod\, 4)$
 $$\xymatrix@-1.9pc{
 &&&&&*{\bullet}&&&&&&&&&&*{\bullet}&&&&&&&&&&*{\bullet}&&&&&&&&&&&&&&&&&&&*{\bullet}\\
 \\
 &*{\bullet}\ar@{-}[rr]_1&&*{\bullet}\ar@{-}[urur]^0\ar@{-}[drdr]_2&&&&*{\bullet}\ar@{-}[lulu]_2\ar@{-}[ldld]^0\ar@{-}[rr]_1&&*{\bullet}\ar@{-}[rr]_2\ar@{-}[dddddd]_0 &&*{\bullet}\ar@{-}[dddddd]^0\ar@{-}[rr]_1 &&*{\bullet}\ar@{-}[urur]^2\ar@{-}[drdr]_0&&&&*{\bullet}\ar@{-}[lulu]_0\ar@{-}[ldld]^2 \ar@{-}[rr]_1 &&*{\bullet}\ar@{-}[rr]_2&&*{\bullet}\ar@{-}[rr]_1&&*{\bullet}\ar@{-}[urur]^0\ar@{-}[drdr]_2&&&&*{\bullet}\ar@{-}[lulu]_2\ar@{-}[ldld]^0\ar@{-}[rr]_1&&*{\bullet}\ar@{-}[rr]_2\ar@{-}[dddddd]_0 &&*{\bullet}\ar@{-}[dddddd]^0\ar@{-}[rr]_1 &&*{\bullet}\ar@{.}[rrrrr]&&&&& 
 *{\bullet}\ar@{-}[rr]_2 &&*{\bullet}\ar@{-}[rr]_1 &&*{\bullet}\ar@{-}[urur]^0\ar@{-}[drdr]_2&&&&*{\bullet}\ar@{-}[lulu]_2\ar@{-}[ldld]^0\ar@{-}[rr]_1&&*{\bullet}\ar@2{-}[dddddd]_0^2\\
 \\
 &&&&&*{\bullet}\ar@{-}[dd]_1&&&&&&&&&&*{\bullet}\ar@{-}[dd]_1&&&&&&&&&&*{\bullet}\ar@{-}[dd]_1&&&&&&&&&&&&&&&&&&&*{\bullet}\ar@{-}[dd]_1\\
 \\
 &&&&&*{\bullet}&&&&&&&&&&*{\bullet}&&&&&&&&&&*{\bullet}&&&&&&&&&&&&&&&&&&&*{\bullet}\\
 \\
  &*{\bullet} \ar@{-}[rr]_1 &&*{\bullet}\ar@{-}[urur]^2\ar@{-}[drdr]_0&&&&*{\bullet}\ar@{-}[lulu]_0\ar@{-}[ldld]^2\ar@{-}[rr]_1&&*{\bullet}\ar@{-}[rr]_2 &&*{\bullet}\ar@{-}[rr]_1 &&*{\bullet}\ar@{-}[urur]^0\ar@{-}[drdr]_2&&&&*{\bullet}\ar@{-}[lulu]_2\ar@{-}[ldld]^0 \ar@{-}[rr]_1 &&*{\bullet}\ar@{-}[rr]_2&&*{\bullet}\ar@{-}[rr]_1&&*{\bullet}\ar@{-}[urur]^2\ar@{-}[drdr]_0&&&&*{\bullet}\ar@{-}[lulu]_0\ar@{-}[ldld]^2\ar@{-}[rr]_1&&*{\bullet}\ar@{-}[rr]_2 &&*{\bullet}\ar@{-}[rr]_1 &&*{\bullet}\ar@{.}[rrrrr]&&&&&
*{\bullet}\ar@{-}[rr]_2 &&*{\bullet}\ar@{-}[rr]_1 &&*{\bullet}\ar@{-}[urur]^2\ar@{-}[drdr]_0&&&&*{\bullet}\ar@{-}[lulu]_0\ar@{-}[ldld]^2\ar@{-}[rr]_1&&*{\bullet}\\
 \\
 &&&&&*{\bullet}&&&&&&&&&&*{\bullet}&&&&&&&&&&*{\bullet}&&&&&&&&&&&&&&&&&&&*{\bullet}\\
 }$$ 
 \end{small}
 \end{center}
 \end{proposition}
%\begin{proof}
%Let $G=\langle \rho_0,\rho_1,\rho_2\rangle$ be the group with one of the given permutation representation
%graph. First we have $(\rho_0\rho_1)^3=(\rho_1\rho_2)^6=(\rho_0\rho_1)^2=\big((\rho_0\rho_1\rho_2)^2\big)^{s}=1$.
%Hence $|G|\leq 12s^2$.
%Consider the vertex $x$ of the graph. 
%The stabilizer $G_x$ of $x$ contains $\rho_0$, $\rho_2$ and $u^2$, which means that $G_x \cong \langle u^2\rangle \rtimes \langle \rho_0 , \rho_2\rangle $,
%thus $|G_x|\geq 2s$, implying that $|G|\geq 12s^2$. Consequently the graphs given in this lemma
%are permutation representations of the automorphism group of the toroidal maps $\{3, 6\}_{(s,0)}$.
%\end{proof}

%-----------------------------------------------------------------------------------------------------------------------------------------------------------------------------------------------------------------------------------------------------------------------------------------------------------
\begin{proposition}
Let  $s\geq 3$. The following graphs are CPR graphs of $\{3,6\}_{(s,0)}$ of degree $6s$. Moreover the stabilizer of a point is, up to a conjugacy, either $ \langle u\rangle \rtimes \langle \rho_0\rho_2\rangle$   or  $\langle v\rangle \rtimes \langle \rho_0\rho_2\rangle$.

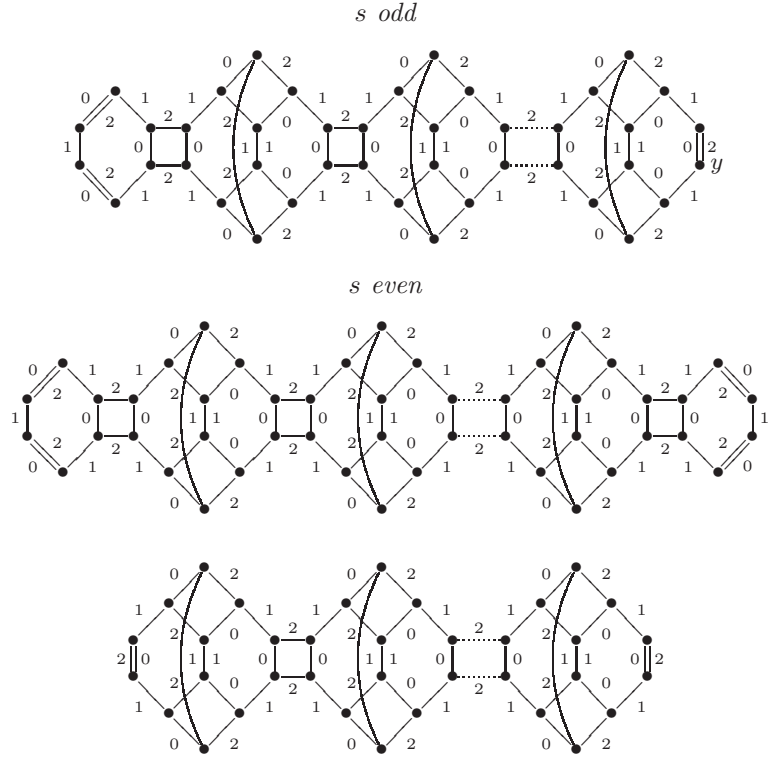
\begin{center}
$s$ odd
\begin{small}
$$\xymatrix@-1.9pc{
&& && && && && *{\bullet}\ar@{-}[drdr]^2 && && && && && *{\bullet}\ar@{-}[drdr]^2 && && && && && & *{\bullet}\ar@{-}[drdr]^2\\
&&\\
&&*{\bullet}\ar@{-}[drdr]^1 && && && *{\bullet}\ar@{-}[urur]^0\ar@{-}[drdr]_2 && && *{\bullet}\ar@{-}[drdr]^1 && && && *{\bullet}\ar@{-}[urur]^0\ar@{-}[drdr]_2 && && *{\bullet} \ar@{-}[drdr]^1 && && && & *{\bullet}\ar@{-}[urur]^0\ar@{-}[drdr]_2 && && *{\bullet} \ar@{-}[drdr]^1 &&\\
&&&&\\
*{\bullet}\ar@{-}[dd]_1\ar@{=}[urur]^0_2 && && *{\bullet}\ar@{-}[dd]_0\ar@{-}[rr]^2 && *{\bullet}\ar@{-}[dd]^0\ar@{-}[urur]^1 && && *{\bullet}\ar@{-}[urur]_0\ar@{-}[dd]^1 && && *{\bullet}\ar@{-}[dd]_0\ar@{-}[rr]^2 && *{\bullet}\ar@{-}[dd]^0\ar@{-}[urur]^1 && && *{\bullet}\ar@{-}[urur]_0\ar@{-}[dd]^1 && && *{\bullet}\ar@{-}[dd]_0\ar@{.}[rrr]^2 &&& *{\bullet}\ar@{-}[dd]^0\ar@{-}[urur]^1 && && *{\bullet}\ar@{-}[urur]_0\ar@{-}[dd]^1 && && *{\bullet}\ar@{=}[dd]_0^2\\
&&&&\\
*{\bullet}\ar@{=}[drdr]_0^2 && &&*{\bullet}\ar@{-}[rr]_2 && *{\bullet}\ar@{-}[drdr]_1 && && *{\bullet}\ar@{-}[drdr]^0 && &&*{\bullet}\ar@{-}[rr]_2 && *{\bullet}\ar@{-}[drdr]_1 && && *{\bullet}\ar@{-}[drdr]^0 && &&*{\bullet}\ar@{.}[rrr]_2 &&& *{\bullet}\ar@{-}[drdr]_1 && && *{\bullet}\ar@{-}[drdr]^0 && && *{\bullet} & *{y}\\
&& &&\\
&&*{\bullet}\ar@{-}[urur]_1 && && &&*{\bullet}\ar@{-}[drdr]_0\ar@{-}[urur]^2 && && *{\bullet} \ar@{-}[urur]_1 && && &&*{\bullet}\ar@{-}[drdr]_0\ar@{-}[urur]^2 && && *{\bullet} \ar@{-}[urur]_1 && & && &&*{\bullet}\ar@{-}[drdr]_0\ar@{-}[urur]^2 && && *{\bullet} \ar@{-}[urur]_1 &&\\
\\
&& && && && && *{\bullet}\ar@{-}[urur]_2 \ar@{-}@/^0.75pc/[uuuuuuuuuu]_1 && && && && && *{\bullet}\ar@{-}[urur]_2 \ar@{-}@/^0.75pc/[uuuuuuuuuu]_1 && && && && && & *{\bullet}\ar@{-}[urur]_2 \ar@{-}@/^0.75pc/[uuuuuuuuuu]_1
 }$$
\end{small}

 $s$ even
 \begin{small}
$$\xymatrix@-1.9pc{
&& && && && && *{\bullet}\ar@{-}[drdr]^2 && && && && && *{\bullet}\ar@{-}[drdr]^2 && && && && && & *{\bullet}\ar@{-}[drdr]^2\\
&&\\
&&*{\bullet}\ar@{-}[drdr]^1 && && && *{\bullet}\ar@{-}[urur]^0\ar@{-}[drdr]_2 && && *{\bullet}\ar@{-}[drdr]^1 && && && *{\bullet}\ar@{-}[urur]^0\ar@{-}[drdr]_2 && && *{\bullet} \ar@{-}[drdr]^1 && && && & *{\bullet}\ar@{-}[urur]^0\ar@{-}[drdr]_2 && && *{\bullet} \ar@{-}[drdr]^1 && && && *{\bullet}\ar@{=}[drdr]^0_2\\
&&&&\\
*{\bullet}\ar@{-}[dd]_1\ar@{=}[urur]^0_2 && && *{\bullet}\ar@{-}[dd]_0\ar@{-}[rr]^2 && *{\bullet}\ar@{-}[dd]^0\ar@{-}[urur]^1 && && *{\bullet}\ar@{-}[urur]_0\ar@{-}[dd]^1 && && *{\bullet}\ar@{-}[dd]_0\ar@{-}[rr]^2 && *{\bullet}\ar@{-}[dd]^0\ar@{-}[urur]^1 && && *{\bullet}\ar@{-}[urur]_0\ar@{-}[dd]^1 && && *{\bullet}\ar@{-}[dd]_0\ar@{.}[rrr]^2 &&& *{\bullet}\ar@{-}[dd]^0\ar@{-}[urur]^1 && && *{\bullet}\ar@{-}[urur]_0\ar@{-}[dd]^1 && && *{\bullet}\ar@{-}[dd]_0 \ar@{-}[rr]^2 && *{\bullet}\ar@{-}[dd]^0 \ar@{-}[urur]^1 && && *{\bullet}\ar@{-}[dd]^1\\
&&&&\\
*{\bullet}\ar@{=}[drdr]_0^2 && &&*{\bullet}\ar@{-}[rr]_2 && *{\bullet}\ar@{-}[drdr]_1 && && *{\bullet}\ar@{-}[drdr]^0 && &&*{\bullet}\ar@{-}[rr]_2 && *{\bullet}\ar@{-}[drdr]_1 && && *{\bullet}\ar@{-}[drdr]^0 && &&*{\bullet}\ar@{.}[rrr]_2 &&& *{\bullet}\ar@{-}[drdr]_1 && && *{\bullet}\ar@{-}[drdr]^0 && && *{\bullet}\ar@{-}[rr]_2 && *{\bullet}\ar@{-}[drdr]_1 && && *{\bullet}\\
&& &&\\
&&*{\bullet}\ar@{-}[urur]_1 && && &&*{\bullet}\ar@{-}[drdr]_0\ar@{-}[urur]^2 && && *{\bullet} \ar@{-}[urur]_1 && && &&*{\bullet}\ar@{-}[drdr]_0\ar@{-}[urur]^2 && && *{\bullet} \ar@{-}[urur]_1 && & && &&*{\bullet}\ar@{-}[drdr]_0\ar@{-}[urur]^2 && && *{\bullet} \ar@{-}[urur]_1 && && && *{\bullet}\ar@{=}[urur]_0^2 \\
\\
&& && && && && *{\bullet}\ar@{-}[urur]_2 \ar@{-}@/^0.75pc/[uuuuuuuuuu]_1 && && && && && *{\bullet}\ar@{-}[urur]_2 \ar@{-}@/^0.75pc/[uuuuuuuuuu]_1 && && && && && & *{\bullet}\ar@{-}[urur]_2 \ar@{-}@/^0.75pc/[uuuuuuuuuu]_1
 }$$

$$\xymatrix@-1.9pc{
 &&& && *{\bullet}\ar@{-}[drdr]^2 && && && && && *{\bullet}\ar@{-}[drdr]^2 && && && && && & *{\bullet}\ar@{-}[drdr]^2\\
\\
 &&& *{\bullet}\ar@{-}[urur]^0\ar@{-}[drdr]_2 && && *{\bullet}\ar@{-}[drdr]^1 && && && *{\bullet}\ar@{-}[urur]^0\ar@{-}[drdr]_2 && && *{\bullet} \ar@{-}[drdr]^1 && && && & *{\bullet}\ar@{-}[urur]^0\ar@{-}[drdr]_2 && && *{\bullet} \ar@{-}[drdr]^1 &&\\
&&&&\\
 &*{\bullet}\ar@{=}[dd]^0_2\ar@{-}[urur]^1 && && *{\bullet}\ar@{-}[urur]_0\ar@{-}[dd]^1 && && *{\bullet}\ar@{-}[dd]_0\ar@{-}[rr]^2 && *{\bullet}\ar@{-}[dd]^0\ar@{-}[urur]^1 && && *{\bullet}\ar@{-}[urur]_0\ar@{-}[dd]^1 && && *{\bullet}\ar@{-}[dd]_0\ar@{.}[rrr]^2 &&& *{\bullet}\ar@{-}[dd]^0\ar@{-}[urur]^1 && && *{\bullet}\ar@{-}[urur]_0\ar@{-}[dd]^1 && && *{\bullet}\ar@{=}[dd]_0^2\\
&&&&&\\
&*{\bullet}\ar@{-}[drdr]_1 && && *{\bullet}\ar@{-}[drdr]^0 && &&*{\bullet}\ar@{-}[rr]_2 && *{\bullet}\ar@{-}[drdr]_1 && && *{\bullet}\ar@{-}[drdr]^0 && &&*{\bullet}\ar@{.}[rrr]_2 &&& *{\bullet}\ar@{-}[drdr]_1 && && *{\bullet}\ar@{-}[drdr]^0 && && *{\bullet}\\
&&& &&\\
 &&&*{\bullet}\ar@{-}[drdr]_0\ar@{-}[urur]^2 && && *{\bullet} \ar@{-}[urur]_1 && && &&*{\bullet}\ar@{-}[drdr]_0\ar@{-}[urur]^2 && && *{\bullet} \ar@{-}[urur]_1 && & && &&*{\bullet}\ar@{-}[drdr]_0\ar@{-}[urur]^2 && && *{\bullet} \ar@{-}[urur]_1 &&\\
&\\
 &&& && *{\bullet}\ar@{-}[urur]_2 \ar@{-}@/^0.75pc/[uuuuuuuuuu]_1 && && && && && *{\bullet}\ar@{-}[urur]_2 \ar@{-}@/^0.75pc/[uuuuuuuuuu]_1 && && && && && & *{\bullet}\ar@{-}[urur]_2 \ar@{-}@/^0.75pc/[uuuuuuuuuu]_1
 }$$
\end{small}
\end{center}
\end{proposition}

\subsection{CPR graphs of small degree for $\{3,6\}_{(s,s)}$}

 We start by giving a CPR graph of minimal degree for  $\{3,6\}_{(s,s)}$ when $s\geq 3$. 

\begin{proposition}\label{36ss9s}
Let $s\geq 3$. The following graphs are CPR graphs for $\{3,6\}_{(s,s)}$ of degree $9s$. Moreover the stabilizer of a point is, up to a conjugacy, $ \langle gh\rangle \rtimes \langle \rho_0,\rho_2\rangle$.\\

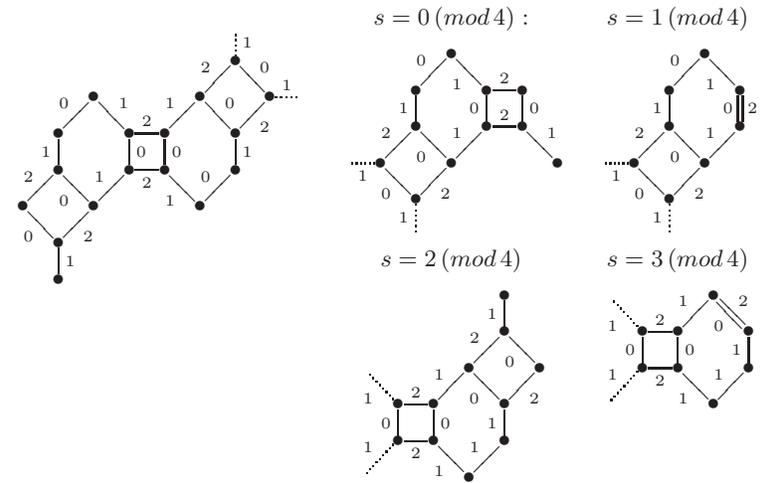
\begin{center}
\begin{small}
\begin{tabular}{ccc}
\multirow{4}{*}{$$\xymatrix@-1.9pc{ 
  && && && && && && \ar@{.}[dd]^1  \\
   \\
  && && && && && &&*{\bullet} \ar@{-}[ddll]_2 \ar@{-}[ddrr]^0 \\
   \\
  && &&*{\bullet} \ar@{-}[ddll]_0 \ar@{-}[ddrr]^1 && && &&*{\bullet} \ar@{-}[ddll]_1 \ar@{-}[ddrr]^0 && &&*{\bullet}\ar@{-}[ddll]^2 \ar@{.}[rr]^1 &&  \\
   \\
  &&*{\bullet} \ar@{-}[dd]_1 && &&*{\bullet} \ar@{-}[dd]^0 \ar@{-}[rr]^2 &&*{\bullet} \ar@{-}[dd]^0  && &&*{\bullet} \ar@{-}[dd]^1\\
   \\
  &&*{\bullet} \ar@{-}[ddll]_2 \ar@{-}[ddrr]_0 && &&*{\bullet} \ar@{-}[ddll]_1 \ar@{-}[rr]_2 &&*{\bullet} \ar@{-}[ddrr]_1 && &&*{\bullet} \ar@{-}[ddll]_0                \\
  \\
 *{\bullet} \ar@{-}[ddrr]_0 && &&*{\bullet} \ar@{-}[ddll]^2 && && &&*{\bullet} \\
   \\
  &&*{\bullet}\ar@{-}[dd]^1 \\
   \\
  &&*{\bullet}\ar@{-}\\
 }$$} & $s= 0\,(mod\,4):$& $s= 1\, (mod \,4)$\\
                  & $$\xymatrix@-1.9pc{
&& && &&*{\bullet} \ar@{-}[ddll]_0 \ar@{-}[ddrr]_1 && && &&\\
&& && && && && &&\\
&& &&*{\bullet} \ar@{-}[dd]_1 && &&*{\bullet} \ar@{-}[dd]_0 \ar@{-}[rr]^2 &&*{\bullet}\ar@{-}[dd]^0 && \\
&& && && && && && \\
&& &&*{\bullet} \ar@{-}[ddll]_2 \ar@{-}[ddrr]_0 && &&*{\bullet} \ar@{-}[ddll]_1 \ar@{-}[rr]^2&&*{\bullet}\ar@{-}[ddrr]^1 && \\
&& && && && && && \\
\ar@{.}[rr]_1&&*{\bullet}\ar@{-}[ddrr]_0 && &&*{\bullet}\ar@{-}[ddll]^2 && && &&*{\bullet} \\
&& && && && && && \\
&& &&*{\bullet}\ar@{.}[dd]_1 && && && && \\
&& && && && && &&\\
&& && && && && && \\
}$$ &  $$\xymatrix@-1.9pc{
&& && &&*{\bullet} \ar@{-}[ddll]_0 \ar@{-}[ddrr]_1 &&  \\
&& && && && \\
&& &&*{\bullet} \ar@{-}[dd]_1 && &&*{\bullet} \ar@{=}[dd]_0^2 \\
&& && && && \\
&& &&*{\bullet} \ar@{-}[ddll]_2 \ar@{-}[ddrr]_0 && &&*{\bullet} \ar@{-}[ddll]_1 \\
&& && && && \\
\ar@{.}[rr]_1&&*{\bullet}\ar@{-}[ddrr]_0 && &&*{\bullet}\ar@{-}[ddll]^2 && \\
&& && && && \\
&& &&*{\bullet}\ar@{.}[dd]_1 && && \\
&& && && && \\
&& && && && \\
}$$\\
                  &  $s= 2\, (mod\, 4)$ & $s= 3\, (mod\, 4) $\\
                  & $$\xymatrix@-1.9pc{
&& && && &&*{\bullet}\ar@{-}[dd]_1 &&\\
&& && && && &&\\
&& && && &&*{\bullet}\ar@{-}[ddll]_2\ar@{-}[ddrr]_0  &&\\
&& && && && &&\\
\ar@{.}[ddrr]_1&& && &&*{\bullet}\ar@{-}[ddll]_1\ar@{-}[ddrr]_0 && &&*{\bullet}\ar@{-}[ddll]^2 \\
&& && && && &&\\
&&*{\bullet}\ar@{-}[rr]^2\ar@{-}[dd]_0 &&*{\bullet}\ar@{-}[dd]^0 && &&*{\bullet}\ar@{-}[dd]_1 &&\\
&& && && && &&\\
&&*{\bullet}\ar@{-}[rr]_2\ar@{.}[ddll]_1 &&*{\bullet}\ar@{-}[ddrr]_1  && &&*{\bullet}\ar@{-}[ddll]_1  &&\\
&& && && && &&\\
&& && &&*{\bullet} && &&\\
}$$ & $$\xymatrix@-1.9pc{
\ar@{.}[ddrr]_1&& && &&*{\bullet}\ar@{-}[ddll]_1\ar@{=}[ddrr]_0^2 && \\
&& && && && \\
&&*{\bullet}\ar@{-}[rr]^2\ar@{-}[dd]_0 &&*{\bullet}\ar@{-}[dd]^0 && &&*{\bullet}\ar@{-}[dd]_1 \\
&& && && && \\
&&*{\bullet}\ar@{-}[rr]_2\ar@{.}[ddll]_1 &&*{\bullet}\ar@{-}[ddrr]_1  && &&*{\bullet}\ar@{-}[ddll]_1  \\
&& && && && \\
&& && &&*{\bullet} && \\
}$$
\end{tabular}
\end{small}
\end{center}
 \end{proposition}
 
% \begin{proof}
 %Let $G=\langle \rho_0,\rho_1,\rho_2\rangle$ be the group with one of the given permutation representation
%graph. First we have $(\rho_0\rho_1)^3=(\rho_1\rho_2)^6=(\rho_0\rho_1)^2=\big(\rho_0(\rho_1\rho_2)^2\big)^{2s}=1$.
%Hence $|G|\leq 36s^2$.
%Consider the vertex $x$ of the graph. 
%The stabilizer $G_x$ of $x$ contains $\rho_0$, $\rho_2$ and $(gh)$, which means that $G_x \cong \langle \rho_2\rangle \times D_{2s}$,
%thus $|G_x|\geq 4s$, implying that $|G|\geq 36s^2$. Consequently the graphs given in this lemma
%are permutation representations of the automorphism group of the toroidal maps $\{3, 6\}_{(s,s)}$.
 %\end{proof}
%-----------------------------------------------------------------------------------------------------------------------------------------------------------------------------------------------------------------------------------------------------------------------------------------------------------

\begin{proposition}\label{36ss92s}
Let $s$ be even. The following graphs show how CPR graphs for $\{3,6\}_{(s,s)}$ of degree $18s$.  Moreover the stabilizer of a point is, up to a conjugacy, $\langle (gh)^2\rangle \rtimes \langle \rho_0,\rho_2\rangle$.

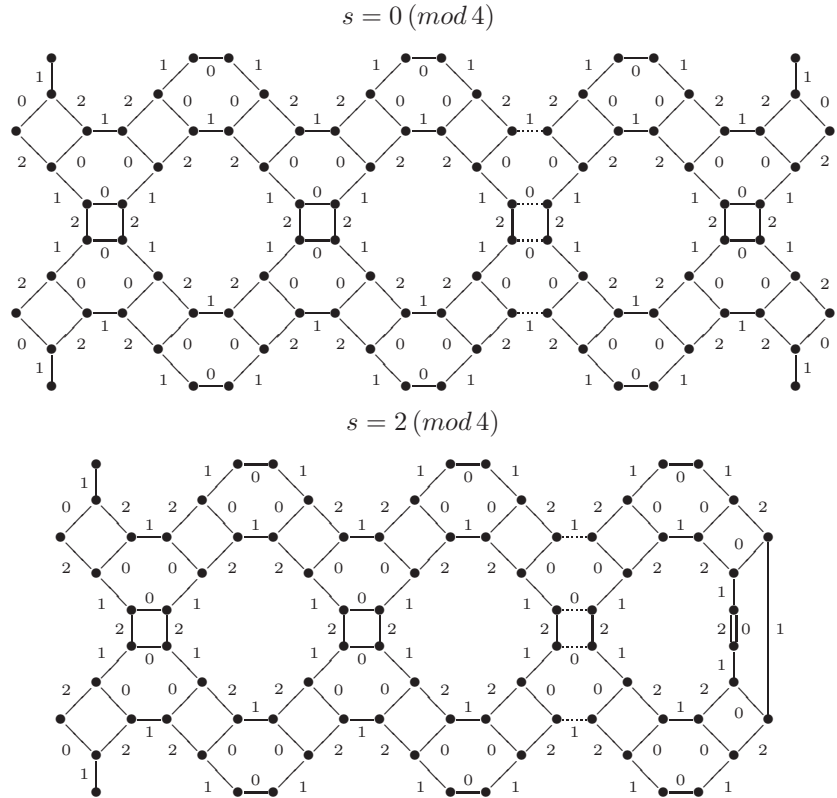
\begin{center}
$s=0 \,(mod\, 4)$
\begin{small}
$$\xymatrix@-1.9pc{
 &&*{\bullet} && && && && *{\bullet}\ar@{-}[rr]_0 && *{\bullet}\ar@{-}[drdr]^1 && && && && && *{\bullet}\ar@{-}[rr]_0 && *{\bullet}\ar@{-}[drdr]^1 && && && && && *{\bullet}\ar@{-}[rr]_0 && *{\bullet}\ar@{-}[drdr]^1 && && && && *{\bullet}\\
 \\
 && *{\bullet}\ar@{-}[uu]^1\ar@{-}[drdr]^2 && && && *{\bullet}\ar@{-}[urur]^1\ar@{-}[drdr]^0 && && && *{\bullet}\ar@{-}[drdr]^2 && && && *{\bullet}\ar@{-}[urur]^1\ar@{-}[drdr]^0 && && && *{\bullet}\ar@{-}[drdr]^2 && && && *{\bullet}\ar@{-}[urur]^1\ar@{-}[drdr]^0 && && && *{\bullet}\ar@{-}[drdr]^2 && && && *{\bullet}\ar@{-}[uu]_1\ar@{-}[drdr]^0\\
 \\
 *{\bullet}\ar@{-}[urur]^0\ar@{-}[drdr]_2 && && *{\bullet} \ar@{-}[rr]^1 && *{\bullet}\ar@{-}[urur]^2 \ar@{-}[drdr]_0 && && *{\bullet}\ar@{-}[rr]^1 && *{\bullet}\ar@{-}[drdr]_2\ar@{-}[urur]^0 && && *{\bullet} \ar@{-}[rr]^1 && *{\bullet}\ar@{-}[urur]^2\ar@{-}[drdr]_0 && && *{\bullet}\ar@{-}[rr]^1 && *{\bullet}\ar@{-}[drdr]_2\ar@{-}[urur]^0 && && *{\bullet} \ar@{.}[rr]^1 && *{\bullet}\ar@{-}[urur]^2\ar@{-}[drdr]_0 && && *{\bullet}\ar@{-}[rr]^1 && *{\bullet}\ar@{-}[drdr]_2\ar@{-}[urur]^0 && &&  *{\bullet} \ar@{-}[rr]^1 && *{\bullet}\ar@{-}[urur]^2\ar@{-}[drdr]_0 && && *{\bullet}\\
 \\
 && *{\bullet}\ar@{-}[urur]_0\ar@{-}[drdr]_1 && && && *{\bullet}\ar@{-}[urur]_2 && && && *{\bullet}\ar@{-}[urur]_0\ar@{-}[drdr]_1 && && && *{\bullet}\ar@{-}[urur]_2 && && && *{\bullet}\ar@{-}[urur]_0\ar@{-}[drdr]_1 && && && *{\bullet}\ar@{-}[urur]_2 && && && *{\bullet}\ar@{-}[urur]_0\ar@{-}[drdr]_1 && && && *{\bullet}\ar@{-}[urur]_2 \\
 \\
 && && *{\bullet}\ar@{-}[rr]^0\ar@{-}[dd]_2 && *{\bullet}\ar@{-}[urur]_1 \ar@{-}[dd]^2 && && && && && *{\bullet} \ar@{-}[rr]^0\ar@{-}[dd]_2 && *{\bullet}\ar@{-}[urur]_1\ar@{-}[dd]^2 && && && && && *{\bullet} \ar@{.}[rr]^0\ar@{-}[dd]_2 && *{\bullet}\ar@{-}[urur]_1 \ar@{-}[dd]^2&& && && && && *{\bullet} \ar@{-}[rr]^0\ar@{-}[dd]_2 && *{\bullet}\ar@{-}[urur]_1\ar@{-}[dd]^2 \\
 \\
 && && *{\bullet}\ar@{-}[rr]_0 && *{\bullet}\ar@{-}[drdr]^1 && && && && && *{\bullet}\ar@{-}[rr]_0 && *{\bullet}\ar@{-}[drdr]^1 && && && && && *{\bullet}\ar@{.}[rr]_0 && *{\bullet}\ar@{-}[drdr]^1 && && && && && *{\bullet}\ar@{-}[rr]_0 && *{\bullet}\ar@{-}[drdr]^1\\
 \\
 && *{\bullet}\ar@{-}[urur]^1 \ar@{-}[drdr]^0 && && && *{\bullet}\ar@{-}[drdr]^2 && && && *{\bullet}\ar@{-}[urur]^1 \ar@{-}[drdr]^0 && && && *{\bullet}\ar@{-}[drdr]^2 && && && *{\bullet}\ar@{-}[urur]^1 \ar@{-}[drdr]^0 && && && *{\bullet}\ar@{-}[drdr]^2 && && && *{\bullet}\ar@{-}[urur]^1 \ar@{-}[drdr]^0&& && && *{\bullet}\ar@{-}[drdr]^2 \\
 \\
 *{\bullet}\ar@{-}[urur]^2 && && *{\bullet} \ar@{-}[rr]_1 && *{\bullet} \ar@{-}[urur]^0 && && *{\bullet} \ar@{-}[rr]^1 && *{\bullet}\ar@{-}[urur]^2 && && *{\bullet} \ar@{-}[rr]_1 && *{\bullet} \ar@{-}[urur]^0 && && *{\bullet} \ar@{-}[rr]^1 && *{\bullet}\ar@{-}[urur]^2 && && *{\bullet} \ar@{.}[rr]_1 && *{\bullet} \ar@{-}[urur]^0 && && *{\bullet} \ar@{-}[rr]^1 &&*{\bullet}\ar@{-}[urur]^2 && && *{\bullet} \ar@{-}[rr]_1 && *{\bullet} \ar@{-}[urur]^0 && && *{\bullet}\\
 \\
 && *{\bullet}\ar@{-}[ulul]^0 \ar@{-}[urur]_2 && && && *{\bullet}\ar@{-}[ulul]^2 \ar@{-}[urur]_0 && && && *{\bullet}\ar@{-}[ulul]^0 \ar@{-}[urur]_2 && && && *{\bullet}\ar@{-}[ulul]^2 \ar@{-}[urur]_0 && && && *{\bullet}\ar@{-}[ulul]^0 \ar@{-}[urur]_2 && && && *{\bullet}\ar@{-}[ulul]^2 \ar@{-}[urur]_0 && && && *{\bullet}\ar@{-}[ulul]^0 \ar@{-}[urur]_2 && && && *{\bullet}\ar@{-}[ulul]^2 \ar@{-}[urur]_0 \\
 \\
 &&*{\bullet}\ar@{-}[uu]^1 && && && && *{\bullet}\ar@{-}[rr]^0\ar@{-}[ulul]^1 && *{\bullet}\ar@{-}[urur]_1 && && && && && *{\bullet}\ar@{-}[rr]^0\ar@{-}[ulul]^1 && *{\bullet}\ar@{-}[urur]_1 && && && && && *{\bullet}\ar@{-}[rr]^0 \ar@{-}[ulul]^1&& *{\bullet}\ar@{-}[urur]_1 && && && && *{\bullet}\ar@{-}[uu]_1
 }$$
 \end{small}
 $s=2\,(mod\, 4)$
 \begin{small}
 $$\xymatrix@-1.9pc{
 &&*{\bullet} && && && && *{\bullet}\ar@{-}[rr]_0 && *{\bullet}\ar@{-}[drdr]^1 && && && && && *{\bullet}\ar@{-}[rr]_0 && *{\bullet}\ar@{-}[drdr]^1 && && && && && *{\bullet}\ar@{-}[rr]_0 && *{\bullet}\ar@{-}[drdr]^1 \\
 \\
 && *{\bullet}\ar@{-}[uu]^1\ar@{-}[drdr]^2 && && && *{\bullet}\ar@{-}[urur]^1\ar@{-}[drdr]^0 && && && *{\bullet}\ar@{-}[drdr]^2 && && && *{\bullet}\ar@{-}[urur]^1\ar@{-}[drdr]^0 && && && *{\bullet}\ar@{-}[drdr]^2 && && && *{\bullet}\ar@{-}[urur]^1\ar@{-}[drdr]^0 && && && *{\bullet}\ar@{-}[drdr]^2 \\
 \\
 *{\bullet}\ar@{-}[urur]^0\ar@{-}[drdr]_2 && && *{\bullet} \ar@{-}[rr]^1 && *{\bullet}\ar@{-}[urur]^2 \ar@{-}[drdr]_0 && && *{\bullet}\ar@{-}[rr]^1 && *{\bullet}\ar@{-}[drdr]_2\ar@{-}[urur]^0 && && *{\bullet} \ar@{-}[rr]^1 && *{\bullet}\ar@{-}[urur]^2\ar@{-}[drdr]_0 && && *{\bullet}\ar@{-}[rr]^1 && *{\bullet}\ar@{-}[drdr]_2\ar@{-}[urur]^0 && && *{\bullet} \ar@{.}[rr]^1 && *{\bullet}\ar@{-}[urur]^2\ar@{-}[drdr]_0 && && *{\bullet}\ar@{-}[rr]^1 && *{\bullet}\ar@{-}[drdr]_2\ar@{-}[urur]^0 && && *{\bullet}\ar@{-}[dddddddddd]^1 \\
 \\
 && *{\bullet}\ar@{-}[urur]_0\ar@{-}[drdr]_1 && && && *{\bullet}\ar@{-}[urur]_2 && && && *{\bullet}\ar@{-}[urur]_0\ar@{-}[drdr]_1 && && && *{\bullet}\ar@{-}[urur]_2 && && && *{\bullet}\ar@{-}[urur]_0\ar@{-}[drdr]_1 && && && *{\bullet}\ar@{-}[urur]_2 && && && *{\bullet}\ar@{-}[urur]^0\ar@{-}[dd]_1  \\
 \\
 && && *{\bullet}\ar@{-}[rr]^0\ar@{-}[dd]_2 && *{\bullet}\ar@{-}[urur]_1 \ar@{-}[dd]^2 && && && && && *{\bullet} \ar@{-}[rr]^0\ar@{-}[dd]_2 && *{\bullet}\ar@{-}[urur]_1\ar@{-}[dd]^2 && && && && && *{\bullet} \ar@{.}[rr]^0\ar@{-}[dd]_2 && *{\bullet}\ar@{-}[urur]_1 \ar@{-}[dd]^2&& && && && *{\bullet} \ar@{=}[dd]_2^0 \\
 \\
 && && *{\bullet}\ar@{-}[rr]_0 && *{\bullet}\ar@{-}[drdr]^1 && && && && && *{\bullet}\ar@{-}[rr]_0 && *{\bullet}\ar@{-}[drdr]^1 && && && && && *{\bullet}\ar@{.}[rr]_0 && *{\bullet}\ar@{-}[drdr]^1 && && && && *{\bullet}\ar@{-}[dd]_1 \\
 \\
 && *{\bullet}\ar@{-}[urur]^1 \ar@{-}[drdr]^0 && && && *{\bullet}\ar@{-}[drdr]^2 && && && *{\bullet}\ar@{-}[urur]^1 \ar@{-}[drdr]^0 && && && *{\bullet}\ar@{-}[drdr]^2 && && && *{\bullet}\ar@{-}[urur]^1 \ar@{-}[drdr]^0 && && && *{\bullet}\ar@{-}[drdr]^2 && && && *{\bullet}\ar@{-}[drdr]_0 \\
 \\
 *{\bullet}\ar@{-}[urur]^2 && && *{\bullet} \ar@{-}[rr]_1 && *{\bullet} \ar@{-}[urur]^0 && && *{\bullet} \ar@{-}[rr]^1 && *{\bullet}\ar@{-}[urur]^2 && && *{\bullet} \ar@{-}[rr]_1 && *{\bullet} \ar@{-}[urur]^0 && && *{\bullet} \ar@{-}[rr]^1 && *{\bullet}\ar@{-}[urur]^2 && && *{\bullet} \ar@{.}[rr]_1 && *{\bullet} \ar@{-}[urur]^0 && && *{\bullet} \ar@{-}[rr]^1 &&*{\bullet}\ar@{-}[urur]^2 && && *{\bullet} \\
 \\
 && *{\bullet}\ar@{-}[ulul]^0 \ar@{-}[urur]_2 && && && *{\bullet}\ar@{-}[ulul]^2 \ar@{-}[urur]_0 && && && *{\bullet}\ar@{-}[ulul]^0 \ar@{-}[urur]_2 && && && *{\bullet}\ar@{-}[ulul]^2 \ar@{-}[urur]_0 && && && *{\bullet}\ar@{-}[ulul]^0 \ar@{-}[urur]_2 && && && *{\bullet}\ar@{-}[ulul]^2 \ar@{-}[urur]_0 && && && *{\bullet}\ar@{-}[ulul]^0 \ar@{-}[urur]_2 \\
 \\
 &&*{\bullet}\ar@{-}[uu]^1 && && && && *{\bullet}\ar@{-}[rr]^0\ar@{-}[ulul]^1 && *{\bullet}\ar@{-}[urur]_1 && && && && && *{\bullet}\ar@{-}[rr]^0\ar@{-}[ulul]^1 && *{\bullet}\ar@{-}[urur]_1 && && && && && *{\bullet}\ar@{-}[rr]^0 \ar@{-}[ulul]^1&& *{\bullet}\ar@{-}[urur]_1
 }$$
  \end{small}
  \end{center}
 \end{proposition}

\section{Acknowledgements}
This research was supported by the Portuguese
Foundation for Science and Technology (FCT- Funda\c c\~ao para a Ci\^encia e Tecnologia),
through CIDMA - Center for Research and Development in Mathematics and
Applications, within project UID/MAT/04106/2019 (CIDMA).

\bibliographystyle{acm}

\end{document}